\documentclass[english,11pt]{amsart}
\usepackage{amssymb,amsmath,amsthm}
\usepackage{srcltx}
\usepackage{babel,graphicx}
\usepackage{color}
\newtheorem{theorem}{Theorem}[section]
\newtheorem{main}{Theorem}
\newtheorem{corollary}[theorem]{Corollary}
\newtheorem{lemma}[theorem]{Lemma}
\newtheorem{proposition}[theorem]{Proposition}

\theoremstyle{definition}
\newtheorem{definition}[theorem]{Definition}

\theoremstyle{remark}
\newtheorem{remark}[theorem]{Remark}
\newtheorem{example}[theorem]{Example}
\numberwithin{equation}{section}

\newcommand{\fin}{\mbox{}\hfill $\square$ \\[0.2cm]}

\newcommand{\eps}{\varepsilon}

\newcommand{\Z}{\mathbf{Z}}

\newcommand{\C}{\mathbf{C}}

\newcommand{\FF}{\mathcal{F}}

\newcommand{\dist}{{\rm dist}}

\newcommand{\Aut}{\mathop{\rm Aut}}
\newcommand{\Fix}{\mathop{\rm Fix}}
\newcommand{\Stab}{\mathop{\rm Stab}}
\newcommand{\Sp}{\mathop{\rm Sp}}
\newcommand{\rsp}{\rho_{\rm sp}}
\newcommand{\dev}{\mathop{\rm dev}\nolimits}
\newcommand{\Norm}{\mathcal{N}}
\newcommand{\Sc}{\mathcal{S}}

\makeatletter
\@namedef{subjclassname@1991}{2020 Mathematics Subject Classification}
\makeatother

\begin{document}

\title[Geometric structures on Hopf manifolds]{Normal forms and geometric structures on Hopf manifolds}
\author{Paul Boureau}

\address{AGM - CNRS - CY Cergy Paris Universit\'e}
\email{paul.boureau@cyu.fr}


\begin{abstract}
We prove that every Hopf manifold of dimension $n\geq2$, primary or secondary, admits a holomorphic $(G,X)$-structure compatible with its complex structure, where $X=\C^n$ and $G$ is generated by the translations and the Guysinsky--Katok group of invertible sub-resonant polynomials. This extends to any dimension a result of B.~McKay and A.~Pokrovskiy, and rests on a self-contained treatment of Berteloot's approach to the Poincar\'e--Dulac normal form. We then study, in the body of the paper, to what extent the structure is unique: marked uniqueness fails, since diagonal Hopf manifolds already carry $n!$ pairwise inequivalent structures, but any two compatible structures differ by a global automorphism of $\C^n$, exactly one structure is aligned, and the compatible structures are classified by a coset space of the sub-resonant group.
\end{abstract}

\maketitle

\tableofcontents

\section{Introduction}\label{s:intro}

\subsection*{Hopf manifolds and the sub-resonant group}
Hopf manifolds are among the most studied compact complex manifolds carrying no K\"ahler metric. Following Kodaira \cite{Kodaira}, a Hopf manifold of dimension $n\geq2$ is a compact complex manifold whose universal cover is biholomorphic to $\C^n\setminus\{0\}$; it is primary when its fundamental group is infinite cyclic, and secondary otherwise. A primary Hopf manifold is a quotient $(\C^n\setminus\{0\})/\langle\gamma\rangle$ by a contraction $\gamma$, an automorphism of $\C^n\setminus\{0\}$ whose iterates converge to $0$ uniformly on compact sets; by Hartogs' theorem $\gamma$ extends to an automorphism of $\C^n$ fixing $0$, and it is a contraction exactly when the eigenvalues of $d_0\gamma$ lie in the punctured unit disc. These manifolds go back to H.~Hopf \cite{Hopf}, who introduced the first examples $(\C^n\setminus\{0\})/\langle z\mapsto 2z\rangle$; a secondary Hopf manifold is a finite quotient of a primary one, the possible finite groups being classified for surfaces by Kato \cite{Kato}. These manifolds have been intensively studied, notably as model spaces for locally conformally K\"ahler and Vaisman metrics; we refer to the recent survey of Istrati and Otiman \cite{IstratiOtiman} and to the monograph \cite{OrneaVerbitsky} for an account and further references.

By the Poincar\'e--Dulac theorem, revisited in Sections \ref{s:sr}--\ref{s:pd}, $\gamma$ is holomorphically conjugate to a polynomial map $h$. Guysinsky and Katok \cite{Katok} singled out the relevant class of polynomials: fixing an upper triangular representative $L$ of $d_0\gamma$, the sub-resonant polynomials relative to $L$ are those whose monomials satisfy prescribed inequalities among the moduli of the eigenvalues, and the invertible ones form a finite-dimensional algebraic group $SR^*(L)$ under composition, into which the normal form $h$ falls. We set
\[
X=\C^n,\qquad G:=\big\langle\, SR^*(L),\ \C^n \,\big\rangle \subset \Aut(\C^n),
\]
the group generated by $SR^*(L)$ and the translations of $\C^n$.

\subsection*{The main theorem}
Given a holomorphic homogeneous $G$-space $X$, a $(G,X)$-structure on a manifold is an atlas of $X$-valued charts whose transition maps belong to $G$. McKay and Pokrovskiy \cite{Mckay} constructed such structures on Hopf surfaces; Ornea and Verbitsky \cite{Verbi} obtained affine structures in the non-resonant case; and Madera \cite{Madera} recently built flat holomorphic Cartan geometries on every complex Hopf manifold by a different route, compared with ours at the end of this introduction. Our main result extends the existence to all dimensions and to all Hopf manifolds at once.

\begin{main}\label{main:existence}
Every Hopf manifold of dimension $n\geq2$ admits a holomorphic $(G,X)$-structure compatible with its complex structure.
\end{main}

The construction is uniform, but the two cases are handled separately in the body. For a primary Hopf manifold, the Poincar\'e--Dulac normalization identifies $M$ with $(\C^n\setminus\{0\})/\langle h\rangle$ (Section \ref{s:dynamics}), and the inclusion $\C^n\setminus\{0\}\hookrightarrow\C^n$ is a developing map with holonomy $h\in SR^*(L)\subset G$; we call the resulting structure the canonical structure. For a secondary manifold $(\C^n\setminus\{0\})/\Gamma$, the deck group $\Gamma$ is a finite extension of a cyclic contraction group $\langle h\rangle$, and we show in Section \ref{s:secondary} that $\Gamma$ lies entirely in $SR^*(L)$ --- because $h$ turns out to be central in $\Gamma$ --- so that the same inclusion is again a developing map.

\subsection*{Non-uniqueness}
The canonical structure is far from unique, and the failure is already visible in the simplest case. Let $M=(\C^n\setminus\{0\})/\langle L\rangle$ with $L=\mathrm{diag}(\lambda_1,\dots,\lambda_n)$, $0<|\lambda_1|<\dots<|\lambda_n|<1$. For each $\sigma\in\mathfrak{S}_n$ the permutation matrix $P_\sigma$ is a developing map, with holonomy $P_\sigma L P_\sigma^{-1}=\mathrm{diag}(\lambda_{\sigma^{-1}(1)},\dots,\lambda_{\sigma^{-1}(n)})$; this is again a diagonal matrix, hence lies in $SR^*(L)$, since the sub-resonance condition constrains the positions of the nonzero coefficients and not their values. The $n!$ resulting structures are pairwise inequivalent (Example \ref{ex:permutations} and Section \ref{s:classification}).

The body of the paper analyses this non-uniqueness for $n\geq2$. Section \ref{s:rigidity} shows that every developing map extends to a global automorphism of $\C^n$, so that two compatible structures differ by a global automorphism of the model. Sections \ref{s:flag}--\ref{s:uniqueness} single out a distinguished aligned structure and prove it to be the unique one of its kind, and Section \ref{s:classification} classifies all compatible structures by a coset space $SR^*(L)\backslash\Norm(h)$. In dimension one the situation is altogether different: a one-dimensional Hopf manifold is an elliptic curve, and it carries a one-parameter family of affine $(G,X)$-structures, so that no uniqueness holds (this is classical, see Gunning \cite{Gunning}); we therefore restrict to $n\geq2$ throughout.

A detailed comparison with the recent work of Madera \cite{Madera}, which reaches closely related results by a converse route, is given in Section \ref{s:madera} at the end of the paper.

 Section \ref{s:sr} presents sub-resonance type polynomials in a block formulation valid for an arbitrary invertible contracting linear part, and establishes the structural results used throughout (block triangular form, characterization of linear sub-resonant maps, degree bound). Section \ref{s:pd} contains the revisited proof of the Poincar\'e--Dulac theorem. Section \ref{s:model} constructs the model group $G$. Section \ref{s:dynamics} establishes the global dynamics of contracting sub-resonant maps, identifies the normalized model, and proves the existence theorem for primary manifolds. Section \ref{s:rigidity} proves the rigidity of developing maps (Proposition \ref{main:rigidity}). Section \ref{s:flag} introduces the holonomy flag and the alignment condition. Section \ref{s:uniqueness} proves the uniqueness of the aligned structure (Proposition \ref{main:aligned}). Section \ref{s:classification} classifies compatible structures (Proposition \ref{main:classification}) and shows that non-uniqueness persists modulo $\Aut(M)$. Section \ref{s:secondary} treats secondary Hopf manifolds and completes the proof of Theorem \ref{main:existence}. Finally, Section \ref{s:madera} compares our results and methods with the recent work of Madera.

Throughout, we have sought to keep the prerequisites to a minimum and to give the arguments in full detail: apart from the normalization procedure of Berteloot, the text is essentially self-contained, and we have favoured completeness of the proofs over brevity.

I sincerely thank Bertrand Deroin, who introduced me to this problem and was of great help through his exceptional availability. I also extend my heartfelt thanks to Matthieu Madera.

\section{Holomorphic $(G,X)-$structures}\label{s:gx}

We briefly review the notion of $(G,X)-$structure introduced by W.~Thurston \cite{Thurston}. Curious readers may refer to \cite{Goldman}.

\begin{definition}\label{d:gx}
Let $G$ be a complex Lie group acting transitively on a manifold $X$. A $(G,X)$-structure on a manifold $M$ is an atlas $\{(U_i, \phi_i)\}_{i \in I}$ with values in $X$ such that for every $i, j \in I$, the transition map $\phi_i \circ \phi_j^{-1}$ is the restriction of an element of $G$.
\end{definition}

\begin{remark}
If $G$ is a complex Lie group and the action of $G$ on $X$ is holomorphic, then any $(G,X)$-structure on a manifold $M$ naturally inherits a complex structure, as the transition maps are holomorphic. Such a $(G,X)$-structure will be called a holomorphic $(G,X)$-structure. If $M$ is already a complex manifold, we say that a holomorphic $(G,X)$-structure on $M$ is compatible if the complex structure it induces coincides with the given one; equivalently, if its charts are biholomorphisms onto their images.
\end{remark}

We now recall a general recipe that provides a canonical way to construct $(G,X)$-structures on quotient spaces under proper and free actions. It encompasses many classical examples, including the structures that will be constructed in this article, which are all of this type.

\begin{proposition} \label{p:structure quotient}
Let $G$ be a Lie group acting transitively on a manifold $X$. Let $U \subseteq X$ be an open subset and $\Gamma$ be a subgroup of $G$ such that the action of $\Gamma$ on $U$ is free and properly discontinuous. Then the quotient manifold $U/\Gamma$ inherits a $(G,X)$-structure.
\end{proposition}
\begin{example}
Consider the $1$-dimensional Hopf manifold defined by~$M = \C^* / \Gamma$, where $\Gamma = \langle \gamma \rangle$ is the cyclic group generated by the dilation $\gamma(z) = 2z$.

In this case, we set $X = \C$ and $G = \text{Aff}(\C)$, the group of complex affine transformations $z \mapsto az + b$. Let $U = \C^* \subset X$. The action of $\Gamma$ on $U$ is free and properly discontinuous. Since $\Gamma$ is a subgroup of $G$ and the action is holomorphic, the quotient $M$ inherits a complex $(G,X)$-structure from $X$ and is diffeomorphic to a real $2$-torus.
\end{example}

Proposition \ref{p:structure quotient} is the mechanism behind all the structures constructed in this paper: for a primary Hopf manifold we shall take $\Gamma=\langle h\rangle$ (Section \ref{s:dynamics}), and for a secondary one the full deck group $\Gamma$ of the primary cover (Section \ref{s:secondary}); in each case $\Gamma$ will be shown to sit inside $G$.

\subsection{Developing map, holonomy and equivalence}\label{ss:devhol}

Let $M$ be a manifold equipped with a holomorphic $(G,X)$-structure, and let $\pi:\widetilde M\to M$ be its universal cover. It is a classical fact (see \cite{Goldman,Thurston}) that the structure is encoded by a developing pair $(\dev,\rho)$: a local biholomorphism
\[
\dev:\widetilde M\longrightarrow X
\]
and a morphism $\rho:\pi_1(M)\to G$ satisfying the equivariance relation $\dev\circ\gamma=\rho(\gamma)\circ\dev$ for every deck transformation $\gamma\in\pi_1(M)$. Such a pair is not canonically attached to the structure: it is well defined only up to the action of $G$ by
\begin{equation}\label{eq:Gaction}
g\cdot(\dev,\rho):=(g\circ\dev,\ g\,\rho\,g^{-1}),\qquad g\in G,
\end{equation}
and conversely every equivariant pair arises from a $(G,X)$-structure. A $(G,X)$-structure is thus the same thing as a $G$-orbit of developing pairs; we call the pairs representatives of the structure, and we say a structure is compatible (with a given complex structure on $M$) when its developing map is a local biholomorphism for that complex structure, a property independent of the representative.

\begin{definition}[marked equivalence]\label{d:equivalence}
Two developing pairs $(\dev,\rho)$ and $(\dev',\rho')$ on the same complex manifold $M$ (with the same marking of $\pi_1(M)$) are equivalent if there exists $g_0\in G$ with
\[
\dev'=g_0\circ\dev\qquad\text{and}\qquad \rho'=g_0\,\rho\,g_0^{-1},
\]
that is, if they lie in the same $G$-orbit under \eqref{eq:Gaction}. A $(G,X)$-structure is an equivalence class for this relation. Throughout, we denote a structure by $\Sc$ and write $\Sc=[\dev,\rho]$ when $(\dev,\rho)$ is a representative.
\end{definition}

\begin{remark}\label{r:faithful}
In our situation the elements of $G$ are polynomial maps of $\C^n$ (Section \ref{s:model}); by the identity principle, an element of $G$ is determined by its restriction to any nonempty open subset of $\C^n$. The second condition in Definition \ref{d:equivalence} then follows from the first, since $\rho(\gamma)=\dev\circ\gamma\circ\dev^{-1}$ on a nonempty open set; a representative is therefore determined by its developing map alone.
\end{remark}

\section{Sub-resonance type polynomials} \label{s: srpoly}\label{s:sr}

In this section, we present the sub-resonance type polynomials introduced by Guysinsky and Katok in \cite{Katok}, in a block formulation valid for an arbitrary invertible contracting linear part.

Let $L \in M_n(\C)$ be an upper triangular invertible matrix whose eigenvalues $\lambda_1, \dots, \lambda_n$, counted with multiplicity, satisfy
\[
0 < |\lambda_1| \leq \dots \leq |\lambda_n| < 1.
\]
Let $\mu_1<\mu_2<\dots<\mu_l$ denote the distinct moduli of the eigenvalues, and for $1\le i\le l$ set
\[
E_i:=\bigoplus_{|\lambda|=\mu_i}E^{\lambda}(L),
\]
where $E^{\lambda}(L)$ is the characteristic space of $L$ associated with the eigenvalue $\lambda$, so that
\[
\C^n = E_1 \oplus \cdots \oplus E_l .
\]
We also introduce the modulus flag of $L$,
\[
V_\bullet:\qquad 0=V_0\subset V_1\subset\dots\subset V_l=\C^n,\qquad V_i:=E_1\oplus\dots\oplus E_i=\bigoplus_{|\lambda|\le\mu_i}E^\lambda(L),
\]
of signature $d(L):=(\dim V_1,\dots,\dim V_l)$. More generally, for any invertible matrix $B$ whose eigenvalues have moduli among $\{\mu_1,\dots,\mu_l\}$, we write $\FF(B)$ for its modulus flag, $\FF(B)_i:=\bigoplus_{|\lambda|\le\mu_i}E^\lambda(B)$.

We fix once and for all an adapted basis $(e_1,\dots,e_n)$ of $\C^n$: a Jordan basis of $L$ ordered by non-decreasing moduli. In this basis $L$ is upper triangular with diagonal entries $\lambda_1,\dots,\lambda_n$, and the basis is adapted to the decomposition $\C^n=E_1\oplus\dots\oplus E_l$. We write vectors $z=(t_1,\dots,t_l)$ in blocks according to this decomposition.

\begin{definition}\label{d:type}
Let $P: \C^n \to \C^n$ be a homogeneous polynomial map fixing~$0$. We say that it is of type $s = (s_1, \dots, s_l)$ if for all $a_1, \dots, a_l \in \C$ and $(t_1, \dots, t_l) \in E_1 \times \cdots \times E_l$, we have:
\[
P(a_1 t_1 + \cdots + a_l t_l) = a_1^{s_1} \cdots a_l^{s_l} P(t_1 + \cdots + t_l).
\]
\end{definition}

Since a polynomial map is a sum of homogeneous terms, we can introduce the following definition.

\begin{definition}\label{d:sr}
A polynomial map $P: \C^n \to \C^n$ fixing $0$ is said to be of sub-resonance type with respect to $L$ if each component $P_i: \C^n \to E_i$ consists only of homogeneous terms of type $s = (s_1, \dots, s_l)$ satisfying
\begin{equation}\label{eq:sr}
\ln \mu_i \leq \sum_j s_j \ln \mu_j.
\end{equation}
\end{definition}

The notions of homogeneous term and type are invariant under linear coordinate changes preserving the decomposition $\C^n=E_1\oplus\dots\oplus E_l$; it follows that the space of sub-resonance type polynomials does not depend on a coordinate system, but only on this decomposition and on the moduli $\mu_1,\dots,\mu_l$. We denote by $SR(L)$ the set of sub-resonance type polynomials with respect to $L$, and by $SR_k(L)$ the subspace of those that are homogeneous of degree $k$ (all types satisfying $\sum_j s_j=k$). When the moduli are pairwise distinct, $l=n$ and one recovers the definition of \cite{Katok}.

\begin{remark}[monomial criterion]\label{r:monomial}
In the adapted coordinates, write $z^I:=z_1^{i_1}\cdots z_n^{i_n}$ for a multi-index $I=(i_1,\dots,i_n)$, and $H_{I,j}:=z^Ie_j$. For $1\le k\le n$ let $b(k)\in\{1,\dots,l\}$ be the block of the coordinate $k$, so that $|\lambda_k|=\mu_{b(k)}$. The monomial map $H_{I,j}$ is homogeneous of type $s(I)$, where $s_b(I):=\sum_{k:\,b(k)=b} i_k$, and it is sub-resonant if and only if
\[
\ln\mu_{b(j)}\ \le\ \sum_{b}s_b(I)\ln\mu_b\ =\ \sum_{k} i_k\ln|\lambda_k| .
\]
\end{remark}

The following structural lemma will be used constantly.

\begin{lemma}[block triangular form]\label{l:blocktriangular}
Let $g\in SR(L)$, written $g=(g_1,\dots,g_l)$ in blocks. Then for every $i$,
\[
g_i(t_1,\dots,t_l)\ =\ C_i\,t_i\ +\ Q_i(t_{i+1},\dots,t_l),
\]
where $C_i\in\mathrm{End}(E_i)$ is linear and $Q_i$ is a polynomial map without constant term. Moreover every type $s$ occurring in $g$ satisfies
\[
\sum_j s_j\ \le\ \frac{\ln\mu_1}{\ln\mu_l}\ =\ \frac{\ln|\lambda_1|}{\ln|\lambda_n|}.
\]
\end{lemma}

\begin{proof}
Let $\beta$ be a term of $g_i$ of type $s=(s_1,\dots,s_l)\neq 0$. Since $\ln\mu_j<0$ for all $j$:

(a) $s_j=0$ for $j<i$. Otherwise $\sum_j s_j\ln\mu_j\le s_j\ln\mu_j\le\ln\mu_j<\ln\mu_i$, contradicting \eqref{eq:sr}.

(b) if $s_i\ge1$ then $s=e_i$. If $s_i\ge2$ then $\sum_j s_j\ln\mu_j\le 2\ln\mu_i<\ln\mu_i$; if $s_i=1$ and $s_j\ge1$ for some $j>i$ (indices $j<i$ being excluded by (a)), then $\sum_j s_j\ln\mu_j\le \ln\mu_i+\ln\mu_j<\ln\mu_i$. In both cases \eqref{eq:sr} is violated.

Thus $g_i$ is the sum of a term of type $e_i$, linear in $t_i$ and independent of the other variables, and of terms of types $s$ with $s_j=0$ for $j\le i$, i.e.\ depending only on $(t_{i+1},\dots,t_l)$.

For the degree bound: since $\ln\mu_j\le\ln\mu_l<0$,
\[
\ln\mu_1\le\ln\mu_i\le\sum_j s_j\ln\mu_j\le\Big(\sum_j s_j\Big)\ln\mu_l,
\]
whence $\sum_j s_j\le \ln\mu_1/\ln\mu_l$.
\end{proof}

We set
\[
c_0(L) := \left\lceil \frac{\ln |\lambda_1|}{\ln |\lambda_n|} \right\rceil .
\]

\begin{corollary}\label{c:degreebound}
Every element of $SR(L)$ has degree at most $c_0(L)$, and $SR_k(L)=\{0\}$ for $k>c_0(L)$.
\end{corollary}

\begin{corollary}\label{c:flag}
Let $g\in SR(L)$. Then:
\begin{enumerate}
\item for every $z\in\C^n$ and every $i$, $Dg(z)(V_i)\subseteq V_i$;
\item a linear map $A$ is sub-resonant with respect to $L$ if and only if $A(V_i)\subseteq V_i$ for every $i$. Consequently, the set of linear parts $\{Dg(0):g\in SR^*(L)\}$ is exactly the parabolic subgroup
\[
P(V_\bullet):=\{A\in GL_n(\C)\ :\ A(V_i)=V_i\ \ \forall i\},
\]
where $SR^*(L)$ denotes the set of elements of $SR(L)$ with invertible derivative at the origin.
\end{enumerate}
\end{corollary}

\begin{proof}
(1) By Lemma \ref{l:blocktriangular}, the Jacobian matrix of $g$ at any point is block upper triangular with respect to the decomposition $\C^n=\bigoplus_i E_i$: $\partial g_i/\partial t_j=0$ for $j<i$. Such a matrix maps $V_k=E_1\oplus\dots\oplus E_k$ into itself.

(2) The direct implication is (1) applied to $A$ (or Lemma \ref{l:blocktriangular} directly). Conversely, if $A(V_j)\subseteq V_j$ for all $j$, the component $A_i$ of $A$ only involves the variables $t_j$ with $j\ge i$; the corresponding types $e_j$, $j\ge i$, satisfy $\ln\mu_i\le\ln\mu_j$, hence \eqref{eq:sr}. Finally, any $A\in P(V_\bullet)$ is itself a (linear) element of $SR^*(L)$, and any linear part of an element of $SR^*(L)$ lies in $P(V_\bullet)$ by (1) together with invertibility and a dimension count.
\end{proof}

Guysinsky and Katok proved that $SR^*(L)$ is a finite-dimensional algebraic group under composition.

\begin{theorem}[M. Guysinsky and A. Katok]\label{katoklie}
$SR^*(L)$ forms a finite-dimensional algebraic group under composition.
\end{theorem}
\begin{proof}
Let us first show that this set is closed under composition. Let $F = (F_1, ..., F_l)$ and $G = (G_1, ..., G_l)$ be two sub-resonant polynomials. We have $F \circ G = (F_1 \circ G, ..., F_l \circ G)$, and for $1 \leq i \leq l$, we must show that the types of the homogeneous terms of $F_i \circ G$ are sub-resonant. Let $Q_i$ be a homogeneous term of $F_i$. There exists $(s_1, ..., s_l)$ satisfying $\ln\mu_i \leq \displaystyle\sum_j s_j\ln\mu_j$ such that for all $a_1, ..., a_l \in \C$ and $(t_1, ..., t_l) \in E_1 \times \cdots \times E_l$,
$$
Q_i(a_1t_1 + ... + a_lt_l) = a_1^{s_1}...a_l^{s_l}Q_i(t_1+...+t_l).
$$
Now let $R_j$ be a homogeneous term of $G_j$ for all $1 \leq j \leq l$. Then there exists $(s_1^{j}, ..., s_l^{j})$ satisfying $\ln\mu_j \leq \displaystyle\sum_k s_k^j\ln\mu_k$ such that for all $a_1, ..., a_l \in \C$ and $(t_1, ..., t_l) \in E_1 \times \cdots \times E_l$,
$$
R_j(a_1t_1 + ... + a_lt_l) = a_1^{s_1^{j}}...a_l^{s_l^{j}}R_j(t_1+...+t_l).
$$
Thus, for $a_1, ..., a_l \in \C$ and $(t_1, ..., t_l) \in E_1 \times \cdots \times E_l$, we have:
\begin{align*}
    Q_i \left( \sum_j R_j(a_1t_1 + ... + a_lt_l)\right) &= Q_i \left( \sum_j a_1^{s_1^{j}}\cdots a_l^{s_l^{j}}R_j(t_1+...+t_l)\right) \\
    &= \left(a_1^{s_1^{1}}\cdots a_l^{s_l^{1}}\right)^{s_1} \cdots \left(a_1^{s_1^{l}}\cdots a_l^{s_l^{l}}\right)^{s_l}Q_i\left( \sum_j R_j(t_1+...+t_l) \right) \\
    &= \left(a_1^{s_1s_1^{1}}\cdots a_l^{s_1s_l^{1}}\right) \cdots \left(a_1^{s_ls_1^{l}}\cdots a_l^{s_ls_l^{l}}\right)Q_i\left( \sum_j R_j(t_1+...+t_l) \right).
\end{align*}
It remains to verify that
$$
\ln\mu_i \leq \sum_k \sum_j s_k^j s_j \ln\mu_k.
$$
We have $\displaystyle\sum_k \sum_j s_k^j s_j \ln\mu_k = \sum_j \sum_k s_k^j s_j \ln\mu_k$, and $\ln\mu_j \leq \displaystyle\sum_k s_k^j\ln\mu_k$ for all $j$ because $R_j$ is sub-resonant. Thus:
$$
\sum_k \sum_j s_k^j s_j \ln\mu_k \geq \sum_j s_j\ln\mu_j.
$$
The result follows from the sub-resonant nature of $Q_i$. Therefore, $SR^*(L)$ is closed under composition. Since this operation is polynomial in the coefficients, it is algebraic.

By Corollary \ref{c:flag}, the linear part of an element $F \in SR^*(L)$ preserves the flag $V_\bullet$, hence so does its inverse, which is therefore also a sub-resonant polynomial.

Now let us show that sub-resonant polynomials relative to $L$ with invertible derivative at $0$ are invertible under composition and that their inverse is also sub-resonant. Let $F$ be such a map. By composing on the right with the inverse $P_1$ of the derivative of $F$ at $0$, we obtain a map $F_1 = Id + S_2 + S_{>2}$, where $S_2$ is a sum of homogeneous polynomials of degree $2$ and $S_{>2}$ is a sum of terms of strictly higher degrees. By construction, the sum of two sub-resonant polynomials relative to $L$ is a sub-resonant polynomial relative to $L$. Thus, the polynomial $P_2 := Id - S_2 \in SR^*(L)$. By composing $F_1$ with $P_2$ on the right, we obtain the polynomial $F_2 := F_1 \circ P_2 = Id - S_2 + S_2 + S_{>2} = Id + S_3 + S_{>3}$. According to Corollary \ref{c:degreebound}, the degree of the maps involved is bounded by $c_0(L)$. It follows that there exists an integer $m$ such that, continuing the previous construction, we have $F_m = Id$.

Thus, we have shown that $F$ is invertible and
$$
F^{-1} = P_1 \circ \cdots \circ P_m,
$$
where each $P_i \in SR^*(L)$. It follows that $F^{-1} \in SR^*(L)$. The operation associating a polynomial in $SR^*(L)$ to its inverse is polynomial in the coefficients, hence it is algebraic.

Finally, $SR^*(L)$ is indeed a finite-dimensional algebraic group, the dimension being finite by the bound on the degree of sub-resonant polynomials.

\end{proof}

\begin{remark}\label{r:noinvert}
The computation at the beginning of the proof of Theorem \ref{katoklie} does not use the invertibility of the derivatives: the space $SR(L)$ itself is stable under sums, under extraction of homogeneous components, and under composition. This will be used repeatedly in Sections \ref{s:dynamics}--\ref{s:classification}.
\end{remark}

\begin{corollary}\label{c:autpol}
Every element of $SR^*(L)$ is a polynomial automorphism of $\C^n$ fixing the origin, whose inverse is again a sub-resonant polynomial. The identities $F\circ F^{-1}=F^{-1}\circ F=\mathrm{id}$, being polynomial, therefore hold on all of $\C^n$.
\end{corollary}

\section{
A theorem of Poincar\'e and Dulac}\label{s:pd}

In this section, we provide a revisited proof of the well-known Poincar\'e--Dulac normal form theorem.
\begin{theorem}\label{t:PD0}
    Let $F : \C^n \to \C^n$ be a holomorphic map whose linear part at the origin $L$ is invertible and contracting. There exists a polynomial $P$ and a local biholomorphism $\phi$ such that $\phi(0) = 0$, $\phi'(0) = \text{Id}$, and $\phi^{-1} \circ F \circ \phi = P$.
\end{theorem}

We base our work on the proof by Berteloot in \cite{Berteloot}. Given a holomorphic transformation $F : B_r \to \C^n$, where $B_r$ denotes the ball of radius $r$ centered at the origin in $\C^n$, it can be expanded in the form $F = \sum_{p \geq 1} H_p$, where $H_p$ is a homogeneous polynomial map of degree $p$. We then formally construct a local biholomorphism by composing transformations designed to eliminate the term of a certain degree in this decomposition. This method requires studying the convergence of the formally constructed coordinate change.

In this section, all these convergence issues are resolved by the following proposition, derived from \cite{Berteloot}.
\begin{proposition}\label{p:conv}
Let $N$ be an automorphism of $\C^n$ fixing the origin, and let its linear part $L := N'(0)$ satisfy $a \|z\| \leq \|L(z)\| \leq A \|z\|$ for some $0 < a \leq A < 1$.

Let $F : B_r \to \C^n$ be a holomorphic map such that
\[
F = N + \sum_{p \geq k} H_p,
\]
where $H_p$ is a homogeneous polynomial map of degree $p$.

Then, if $k > \frac{\ln a}{\ln A}$, the sequence $(N^{-p} \circ F^p)_p$ converges, and its limit defines a local biholomorphism $\phi$ such that $\phi(0) = 0$, $\phi'(0) = \text{Id}$, and $\phi^{-1} \circ F \circ \phi = N$.
\end{proposition}

Consider a holomorphic map $F = L + \sum_{p \geq 2} H_p$, where $L$ is as in Section \ref{s:sr}. We work in the adapted coordinate system fixed there, in which $L$ is upper triangular.

\begin{definition}
We denote by $\mathcal{H}^p$ the $\C$-vector space of $p$-homogeneous polynomial maps from $\C^n$ to $\C^n$. We equip $\mathcal{H}^p$ with the canonical basis $\mathcal{B}^p := \{H_{I,j} = z^I e_j,\ |I| = p,\ 1 \leq j \leq n\}$ of Remark \ref{r:monomial}.
\end{definition}

\begin{definition}
Any holomorphic transformation of the form $\sum_{p \geq k+1} h_p$ where $h_p \in \mathcal{H}^p$ will also be denoted by $o(k)$.
\end{definition}

To apply Proposition \ref{p:conv}, we aim to reduce the problem to a map of the form $N + o(k)$, where $k$ satisfies the hypotheses of the proposition; any integer strictly greater than $c_0(L)$ will do, for a suitable choice of norm.

The proof then consists of cancelling the non-sub-resonant terms of degree greater than or equal to $2$ by successively conjugating $F$ with local biholomorphisms of the form $\psi_p := I + h_p$, where $h_p \in \mathcal{H}^p$ for $p = 2, \dots, c_0(L) + 1$. Let us first specify the effect of such a conjugation:

\begin{proposition}\label{p:cdv}
Let $\psi := I + h$, where $h \in \mathcal{H}^q$. Assume $F = L + S_{q-1} + H_q + o(q)$, where $S_{q-1} \in \mathcal{H}^2 \oplus \dots \oplus \mathcal{H}^{q-1}$ and $H_q \in \mathcal{H}^q$. Then
\[
\psi^{-1} \circ F \circ \psi = L + S_{q-1} + \big[ H_q + L \circ h - h \circ L \big] + o(q).
\]
\end{proposition}

The following family of operators appears:
\[
M_L^q :\ \left\{
\begin{array}{ccl}
\mathcal{H}^q & \longrightarrow & \mathcal{H}^q\\[3pt]
h & \longmapsto & h \circ L - L \circ h.
\end{array}
\right.
\]

The invertibility of $M_L^q$ in the previous proposition would allow us to conjugate $F$ into a holomorphic transformation without terms of degree $q$; in general only the non-sub-resonant part of degree $q$ can be cancelled, as we now explain.

The central idea for studying these operators is the following.

\begin{proposition}\label{p:triang}
$M_L^q$ is triangular on a reordering of $\mathcal{B}^q$.
\end{proposition}
\begin{proof}
Indeed, for $|I| = q$ and $1 \leq j \leq n$, writing $L=(l_{t,k})$ (upper triangular), we have:
\begin{align*}
    M_L^q(H_{I,j})(z_1, ...,z_n) &= (H_{I,j} \circ L)(z_1, ...,z_n) - (L\circ H_{I,j}) (z_1, ...,z_n) \\
    &= H_{I,j}\left( \sum_{t=1}^n \left(\sum_{k=t}^n l_{t,k}z_k\right)e_t\right)-z^I Le_j \\
    &= \left(\sum_{k=1}^n l_{1,k}z_k\right)^{i_1}\cdots\left(l_{n,n}z_n\right)^{i_n}e_j - \sum_{i=1}^{j}l_{i,j}H_{I,i}(z_1, ...,z_n) \\
    &= (\lambda^I - \lambda_j)H_{I,j}(z_1, ...,z_n) + \left(\sum_{I'} \alpha_{I'}z^{I'}\right)e_j - \sum_{i=1}^{j-1}l_{i,j}H_{I,i}(z_1, ...,z_n) \\
    &= (\lambda^I - \lambda_j)H_{I,j}(z_1, ...,z_n) + \sum_{I'} \alpha_{I'}H_{I',j}(z_1, ...,z_n)- \sum_{i=1}^{j-1}l_{i,j}H_{I,i}(z_1, ...,z_n)
\end{align*}

We define an order on the multi-indices such that $I' \ll I$ if $i_n' > i_n$, or if $i_n' = i_n$ and $i_{n-1}' > i_{n-1}$, and so on. This order ensures that all multi-indices appearing in the sum $\sum_{I'} \alpha_{I'} H_{I',j}(z_1, \dots, z_n)$ are strictly less than $I$. We naturally extend this order to the pairs $(I, j)$ by $(I', j') \ll (I, j)$ if $I' \ll I$, or $I' = I$ and $j' < j$.

We have thus shown that $M_L^q$ is upper triangular in this reordering of $\mathcal{B}^q$.
\end{proof}

These operators behave very well with respect to the space of sub-resonant polynomials relative to $L$. Let us first establish the following:

\begin{lemma}\label{l:stable&ante}
For any $r$:
\begin{enumerate}
    \item $SR_r(L)$ is stable under $M_L^r$.
    \item The pre-images of a polynomial in $SR_r(L)$ under $M_L^r$ are all sub-resonant.
\end{enumerate}
\end{lemma}

\begin{proof}
(1) Fix $r$. Let $h$ be a polynomial of degree $r$ that is sub-resonant relative to $L$. It suffices to note that $M_L^r(h) = h \circ L - L \circ h$ is a sum of compositions of sub-resonant maps, hence sub-resonant by Theorem \ref{katoklie} and Remark \ref{r:noinvert} (stability of $SR(L)$ under composition and sums). We used the fact that $L$ is sub-resonant relative to itself, which follows from Corollary \ref{c:flag} since $L$ preserves $V_\bullet$.

(2) To show the second point, consider $h = \sum_{(I,j)} c_{I,j} H_{I,j}$, a polynomial that is not sub-resonant relative to $L$. Then, there necessarily exists $(I,j)$ such that $c_{I,j} \neq 0$ and $H_{I,j}$ is not sub-resonant relative to $L$. Let $(I_0, j_0)$ be maximal for the order of Proposition \ref{p:triang} satisfying the previous condition. We have:
\[
M_L^r(h) = c_{I_0, j_0} (\lambda^{I_0} - \lambda_{j_0}) H_{I_0, j_0} + \sum_{(I,j) \neq (I_0, j_0)} c_{I,j} M_L^r(H_{I,j}) + R,
\]
where
\[
R = c_{I_0, j_0} M_L^r(H_{I_0, j_0}) - c_{I_0, j_0} (\lambda^{I_0} - \lambda_{j_0}) H_{I_0, j_0}.
\]

We claim that $\lambda^{I_0} \neq \lambda_{j_0}$: otherwise, taking moduli, $\ln\mu_{b(j_0)}=\ln|\lambda_{j_0}|=\sum_k i_k\ln|\lambda_k|$, so that $H_{I_0,j_0}$ would be sub-resonant by the monomial criterion of Remark \ref{r:monomial}. Thus, the term $c_{I_0, j_0} (\lambda^{I_0} - \lambda_{j_0}) H_{I_0, j_0}$ is not sub-resonant relative to $L$. According to Proposition \ref{p:triang}, the remainder $R$ does not contain any component in $H_{I_0, j_0}$.

By Proposition \ref{p:triang} again, if $H_{I,j}$ is not sub-resonant, then $M_L^r(H_{I,j})$ does not have a component in $H_{I_0, j_0}$, due to the maximality of $(I_0, j_0)$. Finally, from the first point of the lemma, if $H_{I,j}$ is sub-resonant, then $M_L^r(H_{I,j})$ is also sub-resonant and hence has no component in $H_{I_0, j_0}$. Thus, $\sum_{(I,j) \neq (I_0,j_0)} c_{I,j} M_L^r(H_{I,j})$ has no component in $H_{I_0, j_0}$.

As a result, $M_L^r(h)$ has a non-zero component in $H_{I_0, j_0}$, and therefore $M_L^r(h)$ is not sub-resonant relative to $L$, which concludes the second point.

\end{proof}

It follows the
\begin{proposition}\label{p:surj}
For any $r$, we have
\[
SR_r(L) + \mathrm{Im}(M_L^r) = \mathcal{H}^r.
\]
\end{proposition}

\begin{proof}
According to the previous lemma, since the zero polynomial is clearly sub-resonant, $\ker(M_L^r)$ is included in $SR_r(L)$. By iterating the lemma, we have that $E^0(M_L^r)$, the characteristic subspace of $M_L^r$ associated with the eigenvalue $0$, is included in $SR_r(L)$. It is clear that a characteristic subspace of an endomorphism associated with a non-zero eigenvalue is included in its image. Thus, using
\[
\mathcal{H}^r = E^0(M_L^r) \oplus \bigoplus_{\lambda \neq 0} E^\lambda(M_L^r),
\]
we obtain the proposition.
\end{proof}

We are now in a position to prove the following:

\begin{theorem}\label{t:PDSR}
Let $F : \C^n \to \C^n$ be a holomorphic map whose linear part at the origin $L$ is invertible and contracting. There exists a polynomial $h \in SR^*(L)$ with $h'(0)=L$ and a local biholomorphism $\phi$ such that $\phi(0) = 0$, $\phi'(0) = \mathrm{Id}$, and $\phi^{-1} \circ F \circ \phi = h$.
\end{theorem}
\begin{proof}
    We write $F = L + H_2 + o(2)$ where $H_2 \in \mathcal{H}^2$. According to Proposition \ref{p:surj}, there exist $h_2 \in SR_2(L)$ and $f_2 \in \mathcal{H}^2$ such that $h_2 + M_L^2(f_2) = H_2$. Applying Proposition \ref{p:cdv} to the local biholomorphism $\psi_2 = \mathrm{Id} + f_2$, we obtain:
\[
F_2 = \psi_2^{-1} \circ F \circ \psi_2 = L + h_2 + H_3 + o(3),
\]
where $H_3 \in \mathcal{H}^3$. Iterating this process, for every $3 \leq r \leq c_0(L) + 1$, there exists a polynomial $h_r \in SR_r(L)$ such that conjugation by the local biholomorphism $\psi_r = \mathrm{Id} + f_r$ given by Propositions \ref{p:surj} and \ref{p:cdv} yields:
\[
F_r = \psi_r^{-1} \circ F_{r-1} \circ \psi_r = L + h_2 + \dots + h_r + H_{r+1} + o(r+1).
\]
Note that $h_{c_0(L)+1}=0$ automatically, since $SR_{c_0(L)+1}(L)=\{0\}$ by Corollary \ref{c:degreebound}.

Set $h := L + h_2 + \dots + h_{c_0(L)}$. By construction $h$ is a sub-resonant polynomial relative to $L$ whose derivative at the origin is $L$, which is invertible; hence $h \in SR^*(L)$, and $h$ is an automorphism of $\C^n$ by Corollary \ref{c:autpol}. Moreover $F_{c_0(L)+1}=h+o(c_0(L)+1)$. Since $c_0(L)+2>\ln|\lambda_1|/\ln|\lambda_n|$, we may choose a norm on $\C^n$ and constants $0<a\le A<1$ with $a\|z\|\le\|Lz\|\le A\|z\|$ and $\ln a/\ln A<c_0(L)+2$. Proposition \ref{p:conv}, applied to $N=h$ and $k=c_0(L)+2$, provides a local biholomorphism tangent to the identity conjugating $F_{c_0(L)+1}$ to $h$; composing all the conjugations gives the desired $\phi$.
\end{proof}

\section{The model group}\label{s:model}

Throughout the rest of the paper, $M=(\C^n\setminus\{0\})/\langle\gamma\rangle$ is a primary Hopf manifold: $\gamma$ is a contraction, that is, an automorphism of $\C^n$ fixing $0$ whose iterates converge to $0$ uniformly on compact sets. Let $L$ denote an upper triangular representative of $d_0\gamma$, with eigenvalues $\lambda_1,\dots,\lambda_n$ as in Section \ref{s:sr}, and let $h\in SR^*(L)$ be the Poincar\'e--Dulac normal form of $\gamma$ provided by Theorem \ref{t:PDSR}. By Corollary \ref{c:autpol}, $h$ is a polynomial automorphism of $\C^n$; Section \ref{s:dynamics} identifies $M$ with the normalized model $(\C^n\setminus\{0\})/\langle h\rangle$.

Let us prove the following:

\begin{proposition}\label{p:Ggroup}
$G := \langle SR^*(L), \C^n \rangle$, where $\C^n$ acts by translation, is an algebraic group of finite dimension.
\end{proposition}

\begin{proof}
For $\tau \in \C^n$, let $t_\tau$ denote the map $z \mapsto z + \tau$. We begin by proving
\begin{lemma}\label{l:Gform}
\[
G = \{t_\tau \circ h \mid \tau \in \C^n, h \in SR^*(L)\}.
\]
\end{lemma}
\begin{proof}
We need to verify that for every $(\tau, h) \in \C^n \times SR^*(L)$,
\[
t_{-h(\tau)} \circ h \circ t_\tau \in SR^*(L).
\]

The group structure of $SR^*(L)$ ensures the invertibility of the derivative at $0$. Let us demonstrate the sub-resonant nature of the result. To this end, note that for a fixed $\tau$, the map $h\mapsto t_{-h(\tau)}\circ h\circ t_\tau$, that is $h\mapsto h(\cdot+\tau)-h(\tau)$, is linear in $h$. Since in the adapted coordinates every element of $SR(L)$ is a linear combination of sub-resonant monomials $H_{I,j}$ (Remark \ref{r:monomial}), it suffices to verify the condition for such monomials.

We have:
\[
t_{-H_{I,j}(\tau)} \circ H_{I,j} \circ t_\tau (z)
= \big((z_1 + \tau_1)^{i_1} \cdots (z_n + \tau_n)^{i_n} - \tau_1^{i_1} \cdots \tau_n^{i_n}\big)\, e_j .
\]
Expanding, the homogeneous terms of $t_{-H_{I,j}(\tau)} \circ H_{I,j} \circ t_\tau$ are scalar multiples of the monomials $z^J e_j$ with $J\le I$ componentwise and $J\neq0$. Each of them is sub-resonant: by Remark \ref{r:monomial} and the sub-resonance of $H_{I,j}$,
\[
\ln \mu_{b(j)} \leq \sum_k i_k \ln |\lambda_k| \leq \sum_k j_k \ln |\lambda_k|,
\]
where the second inequality holds because $j_k\le i_k$ and $\ln|\lambda_k|<0$ for every $k$.
\end{proof}
We can now explicitly describe the group law of $G$. Let $(\tau_1, h_1)$ and $(\tau_2, h_2)$ be two elements of $G$. We have:
\[
(\tau_1, h_1) \cdot (\tau_2, h_2) = t_{\tau_1} \circ h_1 \circ t_{\tau_2} \circ h_2,
\]
\[
= t_{\tau_1} \circ t_{- \alpha} \circ t_\alpha \circ h_1 \circ t_{\tau_2} \circ h_2,
\]
\[
= t_{\tau_1 - \alpha} \circ h_3,
\]
\[
= (\tau_1 - \alpha, h_3),
\]
where $\alpha = -h_1(\tau_2)$ and $h_3 = t_\alpha \circ h_1 \circ t_{\tau_2} \circ h_2$. By the lemma and the group structure of $SR^*(L)$, $h_3 \in SR^*(L)$.

From the previous calculation, and since $(t_\tau\circ h)^{-1}=h^{-1}\circ t_{-\tau}$, we obtain:
\[
(\tau, h)^{-1} = \big(h^{-1}(-\tau),\ t_{-h^{-1}(-\tau)} \circ h^{-1} \circ t_{-\tau}\big).
\]

We observe that these relations are algebraic in $\tau_1$, $\tau_2$, $h_1$, and $h_2$.
\end{proof}

The following complement to Lemma \ref{l:Gform} will be used constantly in the study of uniqueness.

\begin{lemma}\label{l:uniquedecomp}
Every $g_0\in G$ can be written in a unique way as $g_0=t_\tau\circ s$ with $\tau\in\C^n$ and $s\in SR^*(L)$; moreover $\tau=g_0(0)$. Consequently:
\begin{enumerate}
\item $\Stab_G(0)=SR^*(L)$;
\item for every $g_0=t_\tau\circ s\in G$ and every $x\in\C^n$, $Dg_0(x)=Ds(x)$ preserves the flag $V_\bullet$.
\end{enumerate}
\end{lemma}

\begin{proof}
Existence is Lemma \ref{l:Gform}. For uniqueness: $s(0)=0$, so $(t_\tau\circ s)(0)=\tau$; hence $\tau=g_0(0)$ is determined, and then $s=t_{-\tau}\circ g_0$. If $g_0(0)=0$ then $\tau=0$ and $g_0=s\in SR^*(L)$, which proves (1). Point (2) follows from $Dt_\tau=\mathrm{id}$ and Corollary \ref{c:flag}, together with Corollary \ref{c:autpol} for the invertibility of $Ds(x)$ at every point.
\end{proof}

\section{Global dynamics, the normalized model, and existence}\label{s:dynamics}

In this section we establish the global dynamics of contracting sub-resonant maps, deduce that $\langle h\rangle$ acts freely and properly discontinuously on $\C^n\setminus\{0\}$, identify $M$ with the normalized model $(\C^n\setminus\{0\})/\langle h\rangle$, and prove the existence theorem (Theorem \ref{main:existence}) for primary Hopf manifolds.

\begin{lemma}[global dynamics]\label{l:dynamics}
Let $g\in SR(L)$ be such that the spectral radius of $L_g:=Dg(0)$ is $<1$. Then:
\begin{enumerate}
\item $g^{p}\to0$ uniformly on compact subsets of $\C^n$;
\item $\Fix(g)=\{0\}$.
\end{enumerate}
In particular these conclusions hold for every $g\in SR^*(L)$ whose linear part is similar to $L$ (for instance $g=h$), as well as for all its powers $g^{m}$, $m\ge1$.
\end{lemma}

\begin{proof}
Write $g$ in the block triangular form of Lemma \ref{l:blocktriangular}: $g_i=C_it_i+Q_i(t_{i+1},\dots,t_l)$ (with $Q_l=0$). The linear part $L_g$ is block upper triangular with diagonal blocks $C_1,\dots,C_l$ (the linear part of $Q_i$ only contributes to strictly upper blocks), so $\Sp(C_i)\subseteq\Sp(L_g)$ and the spectral radius of each $C_i$ is $<1$. Fix $\rho\in(\max_i\rsp(C_i),1)$; by the spectral radius formula there is $C\ge1$ with $\|C_i^{m}\|\le C\rho^{m}$ for all $i,m$.

Let $K\subset\C^n$ be compact; write $z^{(p)}:=g^{p}(z)$ and let us prove by descending induction on $i$ that $\delta_i(p):=\sup_{z\in K}\|t_i^{(p)}(z)\|\to0$.

Case $i=l$. Lemma \ref{l:blocktriangular} gives $t_l^{(p)}=C_l^{p}\,t_l^{(0)}$, and $\|C_l^{p}\|\to0$.

Induction step. Assume $\delta_j(p)\to0$ for all $j>i$. Since $Q_i$ is a polynomial map without constant term, there is $C'>0$ such that $\|Q_i(w)\|\le C'\|w\|$ for $\|w\|\le1$; the sequences $\delta_j(p)$, $j>i$, being bounded and tending to $0$, the quantity
\[
\eps_p:=\sup_{z\in K}\big\|Q_i\big(t_{i+1}^{(p)}(z),\dots,t_l^{(p)}(z)\big)\big\|
\]
is bounded and tends to $0$. The recursion $t_i^{(p+1)}=C_i\,t_i^{(p)}+Q_i\big(t^{(p)}_{>i}\big)$ solves to
\[
t_i^{(p)}=C_i^{p}\,t_i^{(0)}+\sum_{k=0}^{p-1}C_i^{\,p-1-k}\,Q_i\big(t^{(k)}_{>i}\big),
\qquad\text{whence}\qquad
\delta_i(p)\ \le\ C\rho^{p}\sup_K\|t_i\|\ +\ C\sum_{k=0}^{p-1}\rho^{\,p-1-k}\eps_k.
\]
It remains to check that the second term, the discrete convolution
\[
a_p:=C\sum_{k=0}^{p-1}\rho^{\,p-1-k}\eps_k
\]
of the geometric sequence $(\rho^k)$ with the bounded null sequence $(\eps_k)$, tends to $0$. Let $E:=\sup_k\eps_k<\infty$ and fix $\delta>0$; choose $k_0$ with $\eps_k\le\delta$ for all $k\ge k_0$. Splitting the sum at $k_0$, for $p>k_0$ we have
\[
a_p\ =\ C\sum_{k=0}^{k_0-1}\rho^{\,p-1-k}\eps_k\ +\ C\sum_{k=k_0}^{p-1}\rho^{\,p-1-k}\eps_k
\ \le\ C\,E\,k_0\,\rho^{\,p-k_0}\ +\ C\,\delta\sum_{j=0}^{\infty}\rho^{\,j}
\ =\ C\,E\,k_0\,\rho^{\,p-k_0}\ +\ \frac{C\,\delta}{1-\rho},
\]
where in the first sum we bounded each $\eps_k\le E$ and $\rho^{\,p-1-k}\le\rho^{\,p-k_0}$ (as $k\le k_0-1$), and in the second we bounded each $\eps_k\le\delta$ and extended the geometric sum to infinity. The first term tends to $0$ as $p\to\infty$ (since $\rho<1$), so $\limsup_p a_p\le C\delta/(1-\rho)$; as $\delta>0$ was arbitrary, $a_p\to0$. Therefore $\delta_i(p)\to0$, and by descending induction this proves (1).

(2) If $g(q)=q$ then $q=g^{p}(q)\to0$, so $q=0$.

Finally, if $g\in SR^*(L)$ has linear part similar to $L$, then $\rsp(L_g)=\rsp(L)<1$; and $g^{m}\in SR^*(L)$ (Theorem \ref{katoklie}) has linear part $L_g^{m}$, of spectral radius $<1$.
\end{proof}

\begin{corollary}\label{c:properlydiscontinuous}
Let $g\in SR^*(L)$ have linear part similar to $L$ (for instance $g=h$). Then $\langle g\rangle\simeq\Z$ acts freely and properly discontinuously on $\C^n\setminus\{0\}$.
\end{corollary}

\begin{proof}
Freeness. If $g^{m}(z)=z$ with $m\ge1$ and $z\neq0$, then $z\in\Fix(g^{m})=\{0\}$ by Lemma \ref{l:dynamics}, a contradiction; the case $m\le-1$ reduces to this one by applying $g^{-m}$. In particular, $g$ has infinite order.

Proper discontinuity. Let $K\subset\C^n\setminus\{0\}$ be compact and $d:=\dist(0,K)>0$. By Lemma \ref{l:dynamics}, $g^{m}(K)\subset B(0,d/2)$ for $m\ge m_0$, so $g^{m}(K)\cap K=\emptyset$ for $m\ge m_0$; and for $m\le-m_0$, $g^{m}(K)\cap K\neq\emptyset$ is equivalent to $K\cap g^{-m}(K)\neq\emptyset$, which is excluded in the same way. Hence $\{m\in\Z: g^{m}(K)\cap K\neq\emptyset\}$ is finite.
\end{proof}

\begin{proposition}[normalized model]\label{p:normalizedmodel}
There is a biholomorphism
\[
\Phi:\ (\C^n\setminus\{0\})/\langle\gamma\rangle\ \xrightarrow{\ \sim\ }\ (\C^n\setminus\{0\})/\langle h\rangle .
\]
In particular the generator of $\pi_1(M)\simeq\Z$ acts, in the normalized model, by $h$.
\end{proposition}

\begin{proof}
By Corollary \ref{c:properlydiscontinuous}, $M':=(\C^n\setminus\{0\})/\langle h\rangle$ is a complex manifold. Let $\varphi$ be the local conjugation of Theorem \ref{t:PDSR}, so that $\gamma\circ\varphi=\varphi\circ h$ near $0$. Equip $\C^n$ with an adapted norm in which $\|L\|<1$; since $Dh(0)=L$, we have $\|h(z)\|\le c\|z\|$ (with $c<1$) on a ball $B=B(0,r)$ small enough, so that $h(B)\subseteq B$, and we may choose $r$ so small that $\varphi$ is defined and injective on a neighborhood of $\overline B$. By induction, $\gamma^{k}\circ\varphi=\varphi\circ h^{k}$ on $B$ for every $k\ge0$ (the identity at step $k$ together with $h^{k}(B)\subseteq B$ yields step $k+1$).

Construction. Let $u\in \C^n\setminus\{0\}$. Since $\gamma$ is a contraction, $\gamma^{p}(u)\to0$; as $\varphi(B)$ is a neighborhood of $0$, there is an admissible $p\ge0$, i.e.\ with $\gamma^{p}(u)\in\varphi(B)$. Set
\[
\widetilde\Psi(u):=h^{-p}\big(\varphi^{-1}(\gamma^{p}(u))\big)\in\C^n
\]
($h$ is a global automorphism by Corollary \ref{c:autpol}, so $h^{-p}$ is everywhere defined).

Independence of $p$. If $p$ is admissible, set $x:=\varphi^{-1}(\gamma^{p}u)\in B$; then $\gamma^{p+1}u=\gamma(\varphi(x))=\varphi(h(x))$ with $h(x)\in B$: $p+1$ is admissible and
\[
h^{-(p+1)}\big(\varphi^{-1}(\gamma^{p+1}u)\big)=h^{-(p+1)}(h(x))=h^{-p}(x).
\]
Any two admissible integers are compared through the chain of intermediate ones (every integer larger than an admissible one is admissible), so $\widetilde\Psi$ is well defined.

Properties. Locally, one and the same $p$ is admissible on a neighborhood of $u$ (continuity, openness of $\varphi(B)$), so $\widetilde\Psi$ is holomorphic and, as a composition of local biholomorphisms, is a local biholomorphism. If $\widetilde\Psi(u)=0$ then $\varphi^{-1}(\gamma^{p}u)=h^{p}(0)=0$, hence $\gamma^{p}(u)=0$ and $u=0$ (injectivity of $\gamma$), a contradiction: $\widetilde\Psi$ takes values in $\C^n\setminus\{0\}$. Finally, if $p$ is admissible for $\gamma(u)$ then
\[
\widetilde\Psi(\gamma(u))=h^{-p}\varphi^{-1}(\gamma^{p+1}u)=h\cdot h^{-(p+1)}\varphi^{-1}(\gamma^{p+1}u)=h(\widetilde\Psi(u)) :
\]
$\widetilde\Psi$ is equivariant and descends to a holomorphic map $\Phi:M\to M'$, a local biholomorphism.

Injectivity. Suppose $\widetilde\Psi(u)=h^{m}\,\widetilde\Psi(u')$. Choose $p$ admissible for $u$ and $p'$ for $u'$, and set $x=\varphi^{-1}(\gamma^{p}u)$, $x'=\varphi^{-1}(\gamma^{p'}u')\in B$; the relation reads $x=h^{a}(x')$ with $a:=p-p'+m$. Up to exchanging $(u,p)$ and $(u',p')$, assume $a\ge0$; then $h^{a}(x')\in B$ and
\[
\gamma^{p}(u)=\varphi(h^{a}(x'))=\gamma^{a}(\varphi(x'))=\gamma^{a+p'}(u').
\]
The iterates of $\gamma$ being injective, we deduce that $u$ and $u'$ lie on the same $\gamma$-orbit, so $[u]=[u']$ in $M$. Hence $\Phi$ is injective.

Surjectivity. Let $z\in\C^n\setminus\{0\}$. By Lemma \ref{l:dynamics}, $h^{p}(z)\in B\setminus\{0\}$ for some $p\ge0$; set $u:=\varphi(h^{p}(z))\in \C^n\setminus\{0\}$. The integer $0$ is admissible for $u$ and $\widetilde\Psi(u)=\varphi^{-1}(u)=h^{p}(z)$, so $\Phi([u])=[z]$.

$\Phi$ is therefore a biholomorphism. For $n\ge2$, $\C^n\setminus\{0\}$ is simply connected, hence it is the universal cover of $M'\simeq M$, with deck group $\langle h\rangle\simeq\Z$ (Corollary \ref{c:properlydiscontinuous} gives the infinite order).
\end{proof}

From now on, we identify $M=(\C^n\setminus\{0\})/\langle h\rangle$ via Proposition \ref{p:normalizedmodel}, and for $n\geq 2$ we let $\gamma$ denote the generator of $\pi_1(M)$ acting by $h$. The canonical structure $\Sc_{\mathrm{can}}$ is given by $\dev_{\mathrm{can}}=\mathrm{incl}:\C^n\setminus\{0\}\hookrightarrow\C^n$ and $\rho_{\mathrm{can}}(\gamma)=h$.

\begin{proof}[Proof of Theorem \ref{main:existence}, primary case]
By Proposition \ref{p:normalizedmodel}, $M$ is biholomorphic to $(\C^n\setminus\{0\})/\langle h\rangle$. By Corollary \ref{c:properlydiscontinuous}, the group $\langle h\rangle\subseteq SR^*(L)\subseteq G$ acts freely and properly discontinuously on the open subset $\C^n\setminus\{0\}$ of $X=\C^n$, and $G$ acts transitively and holomorphically on $X$ (it contains the translations). By Proposition \ref{p:structure quotient}, the quotient inherits a holomorphic $(G,X)$-structure; its charts are local inverses of the covering map composed with the inclusion into $\C^n$, hence biholomorphisms for the complex structure of $M$: the structure is compatible. The secondary case is treated in Section \ref{s:secondary}.
\end{proof}

\section{Analytic rigidity of developing maps}\label{s:rigidity}

We assume $n\geq2$ throughout the remainder of the paper. The non-uniqueness announced in the introduction takes its simplest form in the following example, which we will use as a running illustration.

\begin{example}[the permutation family]\label{ex:permutations}
Let $h=L=\mathrm{diag}(\lambda_1,\dots,\lambda_n)$ with $0<|\lambda_1|<\dots<|\lambda_n|<1$, and $M=(\C^n\setminus\{0\})/\langle L\rangle$. For $\sigma\in\mathfrak{S}_n$, let $P_\sigma$ denote the associated permutation matrix. Then
\[
\dev\nolimits_\sigma:=P_\sigma|_{\C^n\setminus\{0\}},\qquad
\rho_\sigma(\gamma):=P_\sigma L P_\sigma^{-1}=\mathrm{diag}(\lambda_{\sigma^{-1}(1)},\dots,\lambda_{\sigma^{-1}(n)})
\]
is a representative of a compatible $(G,X)$-structure on $M$: indeed $\rho_\sigma(\gamma)$ is a diagonal matrix, hence belongs to $SR^*(L)$, the linear sub-resonance condition constraining only the positions of the nonzero coefficients and never their values. We will see (Proposition \ref{main:classification}) that these $n!$ structures are pairwise inequivalent, and even pairwise inequivalent modulo $\Aut(M)$ when $L$ is non-resonant.
\end{example}

The first structural result asserts that every developing map is the restriction of a global automorphism of $\C^n$.

\begin{proposition}\label{main:rigidity}
Let $n\geq 2$ and let $\Sc$ be a compatible $(G,X)$-structure on $M$, with representative $(\dev,\rho)$. Then:
\begin{enumerate}
\item $\dev$ extends uniquely to an entire map $\widehat{\dev}:\C^n\to\C^n$, which is a local biholomorphism at every point;
\item $p:=\widehat{\dev}(0)$ is the unique fixed point of $\rho(\gamma)$, and $g:=t_{-p}\circ\rho(\gamma)\circ t_{p}\in SR^*(L)$;
\item $\Psi:=t_{-p}\circ\widehat{\dev}$ is a holomorphic automorphism of $\C^n$, with $\Psi(0)=0$ and $\Psi\circ h\circ\Psi^{-1}=g$.
\end{enumerate}
In particular \textup{(}coarse uniqueness\textup{)}, any two compatible structures on $M$ have developing maps differing by post-composition with a global holomorphic automorphism of $\C^n$, which conjugates their holonomies.
\end{proposition}

\begin{proof}[Proof of Proposition \ref{main:rigidity}]
Let $(\dev,\rho)$ be a compatible structure: $\dev:\C^n\setminus\{0\}\to\C^n$ is a local biholomorphism and $\dev\circ h=\rho(\gamma)\circ\dev$.

(1) Extension. The components of $\dev$ are holomorphic on $\C^n\setminus\{0\}$ with $n\ge2$; by Hartogs' extension theorem they extend (uniquely) to $\C^n$. Both sides of the identity $\widehat{\dev}\circ h=\rho(\gamma)\circ\widehat{\dev}$ are holomorphic on $\C^n$ and coincide on the dense open set $\C^n\setminus\{0\}$; they therefore coincide everywhere.

Jacobian. Let $J:=\det D\widehat{\dev}\in\mathcal O(\C^n)$. Since $\dev$ is a local biholomorphism, $J$ does not vanish on $\C^n\setminus\{0\}$. Suppose $J(0)=0$: then $1/J$ is holomorphic on $B(0,1)\setminus\{0\}$ and extends holomorphically to $B(0,1)$ by Hartogs ($n\ge2$); the identity $J\cdot(1/J)=1$, valid on $B(0,1)\setminus\{0\}$, extends by continuity and contradicts $J(0)=0$. Hence $J(0)\neq0$ and $\widehat{\dev}$ is a local biholomorphism at every point of $\C^n$.

(2) Fixed point and recentering. Let $p:=\widehat{\dev}(0)$. Evaluating the equivariance at $0$ gives
\[
\rho(\gamma)(p)=\rho(\gamma)\big(\widehat{\dev}(0)\big)=\widehat{\dev}(h(0))=\widehat{\dev}(0)=p .
\]
By Lemma \ref{l:uniquedecomp}, $g:=t_{-p}\,\rho(\gamma)\,t_{p}\in G$ fixes $0$, so $g\in SR^*(L)$. Set $\Psi:=t_{-p}\circ\widehat{\dev}$; then $\Psi(0)=0$ and
\begin{equation}\label{eq:recentered}
\Psi\circ h=g\circ\Psi\qquad\text{on }\C^n .
\end{equation}
Differentiating \eqref{eq:recentered} at $0$: $A\,L=L_g\,A$, where $A:=D\Psi(0)$ is invertible by the Jacobian step and $L_g:=Dg(0)$; hence $L_g=A\,L\,A^{-1}$ is similar to $L$ and $\rsp(L_g)<1$. Lemma \ref{l:dynamics} applied to $g$ gives $\Fix(g)=\{0\}$, whence $\Fix(\rho(\gamma))=t_p(\Fix(g))=\{p\}$: the fixed point is unique.

(3) Bijectivity of $\Psi$. Note that $h$ and $g$ are automorphisms of $\C^n$ (Corollary \ref{c:autpol}), and that \eqref{eq:recentered} gives, for every $k\in\Z$,
\begin{equation}\label{eq:iterated}
\Psi\circ h^{k}=g^{k}\circ\Psi .
\end{equation}
Let $U_0$ be an open neighborhood of $0$ on which $\Psi$ is injective (local biholomorphism at $0$).

Injectivity. If $\Psi(x)=\Psi(y)$, then by \eqref{eq:iterated}, $\Psi(h^{p}x)=g^{p}\Psi(x)=g^{p}\Psi(y)=\Psi(h^{p}y)$ for every $p\ge0$. By Lemma \ref{l:dynamics} applied to $h$, we have $h^{p}x,h^{p}y\in U_0$ for $p$ large; the injectivity of $\Psi$ on $U_0$ gives $h^{p}x=h^{p}y$, hence $x=y$.

Surjectivity. $\Psi(U_0)$ is an open neighborhood of $0$. Let $w\in\C^n$; by Lemma \ref{l:dynamics} applied to $g$, $g^{p}(w)\in\Psi(U_0)$ for some $p\ge0$: $g^{p}(w)=\Psi(x)$, $x\in U_0$. Then, by \eqref{eq:iterated} with $k=-p$,
\[
w=g^{-p}(\Psi(x))=\Psi\big(h^{-p}(x)\big).
\]
$\Psi$ is therefore a biholomorphism of $\C^n$ onto $\C^n$.

For the last assertion (coarse uniqueness), let $(\dev_1,\rho_1)$ and $(\dev_2,\rho_2)$ be two compatible structures and $\dev_i=t_{p_i}\circ\Psi_i$ on $\C^n\setminus\{0\}$ the corresponding decompositions; set $\Theta:=t_{p_2}\circ\Psi_2\circ\Psi_1^{-1}\circ t_{-p_1}\in\Aut(\C^n)$. Then $\dev_2=\Theta\circ\dev_1$, and the relation between holonomies follows from $\rho_i(\gamma)=\widehat{\dev_i}\circ h\circ\widehat{\dev_i}^{\,-1}$.
\end{proof}

\begin{remark}
The coarse uniqueness statement is already a strong rigidity result: a priori, two developing maps could have differed by a wild, non-extendable local biholomorphism; in fact they differ by a global automorphism of the model. But this statement forgets the group $G$: the group $\Aut(\C^n)$ is huge and does not preserve $G$, so this notion of equivalence identifies structures that the geometry $(G,X)$ really distinguishes, such as all the structures of Example \ref{ex:permutations}. It says that the geometry of $M$ is unique up to isomorphism of geometries. The right notion of uniqueness of the structure must be finer; this is the object of the next section.
\end{remark}

\section{The holonomy flag and aligned structures}\label{s:flag}

\subsection{The model flag is intrinsic}

The following linear algebra lemma shows that the flag $V_\bullet$ is an intrinsic invariant of the geometry $(G,X)$: by Lemma \ref{l:uniquedecomp} and Corollary \ref{c:flag}, the differentials at $0$ of the elements of the isotropy group $\Stab_G(0)=SR^*(L)$ form exactly the parabolic subgroup $P(V_\bullet)$, and:

\begin{lemma}[parabolic lemma]\label{l:parabolic}
The subspaces of $\C^n$ invariant under all elements of $P(V_\bullet)$ are exactly $0,V_1,\dots,V_l$. Consequently, $V_\bullet$ is the unique flag of signature $d(L)$ invariant under $P(V_\bullet)$.
\end{lemma}

\begin{proof}
The $V_i$ are invariant by definition. Conversely, let $W\neq 0$ be a subspace invariant under every element of $P(V_\bullet)$.

Step 1: $W$ is invariant under every endomorphism $\varphi$ preserving the flag. Let $\varphi$ satisfy $\varphi(V_i)\subseteq V_i$ for all $i$. The polynomial $t\mapsto\det(\mathrm{id}+t\varphi)$ equals $1$ at $t=0$, so $\mathrm{id}+t\varphi$ is invertible for small $t\neq0$, and preserves each $V_i$ (with equality by dimension): $\mathrm{id}+t\varphi\in P(V_\bullet)$. Then $(\mathrm{id}+t\varphi)(W)\subseteq W$ and $\mathrm{id}(W)\subseteq W$ give $t\varphi(W)\subseteq W$, hence $\varphi(W)\subseteq W$.

Step 2. Let $w\in W\setminus\{0\}$ and let $i$ be minimal with $w\in V_i$ (so $w\notin V_{i-1}$). Let $v\in V_i$ be arbitrary. Since $w\notin V_{i-1}$, there is a linear form $\xi$ with $\xi(w)=1$ and $\xi|_{V_{i-1}}=0$. The endomorphism of rank $\le1$
\[
\varphi:=\xi(\cdot)\,v
\]
preserves the flag: for $k<i$, $\xi(V_k)\subseteq\xi(V_{i-1})=0$; for $k\ge i$, $\varphi(V_k)\subseteq\C v\subseteq V_i\subseteq V_k$. By Step 1, $v=\varphi(w)\in W$. Hence $V_i\subseteq W$.

Step 3. Let $i^{*}$ be the largest index $i$ such that there exists $w\in W$ with $w\in V_i\setminus V_{i-1}$. By Step 2, $V_{i^{*}}\subseteq W$; and by maximality $W\subseteq V_{i^{*}}$. So $W=V_{i^{*}}$.

Finally, an invariant flag consists of invariant subspaces, hence of some of the $V_i$'s; if it has signature $d(L)$, it is $V_\bullet$.
\end{proof}

\subsection{The holonomy flag and the alignment condition}

By Proposition \ref{main:rigidity}, any representative $(\dev,\rho)$ of a compatible structure $\Sc$ determines the following data: the unique fixed point $p$ of $\rho(\gamma)$, and the differential $D\rho(\gamma)(p)$, which is similar to $L$ (it equals $L_g=ALA^{-1}$ in the notation of the proof, since $\rho(\gamma)=t_p\,g\,t_{-p}$ and $D\rho(\gamma)(p)=Dg(0)=L_g$). Its eigenvalues therefore have moduli $\mu_1<\dots<\mu_l$ with the same multiplicities as those of $L$, and its modulus flag has signature $d(L)$. These data depend on the chosen representative, but the invariant we extract from them will not (Lemma \ref{l:aligned-invariant}).

\begin{definition}[holonomy flag; alignment]\label{d:aligned}
Given a representative $(\dev,\rho)$ of a compatible structure $\Sc$, with $p$ the unique fixed point of $\rho(\gamma)$, its holonomy flag is the modulus flag of $D\rho(\gamma)(p)$,
\[
\FF(\dev,\rho)\ :=\ \FF\big(D\rho(\gamma)(p)\big),\qquad \FF(\dev,\rho)_i=\bigoplus_{|\lambda|\le\mu_i}E^\lambda\big(D\rho(\gamma)(p)\big)\ \subset\ \C^n\simeq T_p\C^n .
\]
We say the structure $\Sc$ is aligned if some (equivalently, by Lemma \ref{l:aligned-invariant}, any) representative has holonomy flag equal to $V_\bullet$.
\end{definition}

\begin{lemma}[invariance]\label{l:aligned-invariant}
For representatives $(\dev,\rho)$ and $(\dev',\rho')$ in the same $G$-orbit, $\FF(\dev,\rho)=V_\bullet$ if and only if $\FF(\dev',\rho')=V_\bullet$. Hence alignment is a well-defined property of the structure $\Sc$. Moreover, the canonical structure is aligned.
\end{lemma}

\begin{proof}
It suffices to treat $(\dev',\rho')=(g_0\circ\dev,\,g_0\rho g_0^{-1})$ for $g_0\in G$. The fixed point of $\rho'(\gamma)=g_0\rho(\gamma)g_0^{-1}$ is $p'=g_0(p)$ and, by the chain rule,
\[
D\rho'(\gamma)(p')=Dg_0(p)\cdot D\rho(\gamma)(p)\cdot Dg_0(p)^{-1}.
\]
Conjugation by an invertible matrix $B$ transports characteristic spaces ($E^\lambda(BTB^{-1})=B\,E^\lambda(T)$), hence modulus flags: $\FF(\dev',\rho')=Dg_0(p)\cdot\FF(\dev,\rho)$. Now $Dg_0(p)$ preserves $V_\bullet$ by Lemma \ref{l:uniquedecomp}; thus $\FF(\dev,\rho)=V_\bullet$ if and only if $\FF(\dev',\rho')=Dg_0(p)(V_\bullet)=V_\bullet$.

For the canonical structure, the representative $(\mathrm{incl},\rho_{\mathrm{can}})$ has $p=0$, $\rho_{\mathrm{can}}(\gamma)=h$, $Dh(0)=L$ and $\FF(L)=V_\bullet$ by definition.
\end{proof}

\section{Uniqueness of the aligned structure}\label{s:uniqueness}

The technical core is the following rigidity lemma. It corrects the naive rigidity statement, which is falsewithout a hypothesis on $L_g$ (as the linear counterexample $\psi=P_\sigma$ already shows), by identifying the hypothesis that makes it true, namely the position of the flag.

We first record two elementary jet extraction identities.

\begin{lemma}[jet extraction]\label{l:jets}
Let $u=\sum_{j\ge1}u_j$ and $v=\sum_{m\ge1}v_m$ be holomorphic germs at $0$ fixing $0$, written in homogeneous components, and let $k\ge2$. Write $v^{(k-1)}:=\sum_{m\le k-1}v_m$ and $u^{(k-1)}:=\sum_{j\le k-1}u_j$ for the truncations, and $[\,\cdot\,]_k$ for the degree-$k$ homogeneous component. Then
\begin{align}
[u\circ v]_k&=u_1\circ v_k\ +\ \big[(u-u_1)\circ v^{(k-1)}\big]_k, \label{eq:jet1}\\
[u\circ v]_k&=u_k\circ v_1\ +\ \big[u^{(k-1)}\circ v\big]_k. \label{eq:jet2}
\end{align}
\end{lemma}

\begin{proof}
By polarization, $[u\circ v]_k=\sum_{j\ge1}\ \sum_{m_1+\dots+m_j=k}\widetilde u_j(v_{m_1},\dots,v_{m_j})$, where $\widetilde u_j$ is the symmetric $j$-multilinear form of $u_j$ and $m_i\ge1$. For \eqref{eq:jet1}: the $j=1$ term contributes $u_1\circ v_k$ in degree $k$; and for $j\ge2$, no $m_i$ can be $\ge k$ (otherwise the sum of the $j\ge2$ degrees would be $\ge k+1$), so $v$ may be replaced by $v^{(k-1)}$. For \eqref{eq:jet2}: the term $j=k$ only contributes through $\widetilde u_k(v_1,\dots,v_1)=u_k\circ v_1$ (as soon as one argument is some $v_m$ with $m\ge2$, the degree is $\ge k+1$); the terms $j>k$ have degree $>k$; the terms $j\le k-1$ constitute $[u^{(k-1)}\circ v]_k$.
\end{proof}

\begin{lemma}[aligned rigidity]\label{l:alignedrigidity}
Let $h,g\in SR^*(L)$ with $Dh(0)=L$ and $\FF(L_g)=V_\bullet$, where $L_g:=Dg(0)$. Let $\psi$ be a local biholomorphism at $0$ such that $\psi(0)=0$ and
\[
\psi\circ h=g\circ\psi\qquad\text{near }0 .
\]
Then $\psi\in SR^*(L)$; it is therefore polynomial of degree $\le c_0(L)$ and extends to an automorphism of $\C^n$.
\end{lemma}

\begin{proof}
Write $\psi=\sum_{k\ge1}\psi_k$ for the Taylor expansion in homogeneous components (convergent near $0$), $A:=\psi_1=D\psi(0)$, and similarly $h=L+\sum_{j\ge2}h_j$, $g=L_g+\sum_{j\ge2}g_j$, where $h_j,g_j\in SR_j(L)$ (the homogeneous components of an element of $SR(L)$ are sub-resonant, by definition).

Step 1: linear part. In degree $1$, the equivariance gives $A\,L=L_g\,A$ with $A$ invertible. Conjugation transports characteristic spaces: $A(E^\lambda(L))=E^\lambda(L_g)$ for every eigenvalue $\lambda$. Summing over $|\lambda|\le\mu_i$:
\[
A(V_i)=\bigoplus_{|\lambda|\le\mu_i}E^\lambda(L_g)=\FF(L_g)_i=V_i,
\]
by the alignment hypothesis. Hence $A\in P(V_\bullet)$, and by Corollary \ref{c:flag}, $A$ and $A^{-1}$ are linear sub-resonant maps: $\psi_1=A\in SR_1(L)$. Note also that $A L A^{-1}=L_g$ is equivalent to
\begin{equation}\label{eq:AL}
A^{-1}\,L_g\,A=L .
\end{equation}

Step 2: induction. We show by strong induction that $\psi_k\in SR_k(L)$ for every $k\ge1$; the case $k=1$ is Step 1. Let $k\ge2$ and suppose $\psi_m\in SR_m(L)$ for $m<k$; in particular the truncation $\psi^{(k-1)}=\sum_{m\le k-1}\psi_m$ belongs to $SR(L)$.

Let us extract the degree-$k$ component of $\psi\circ h=g\circ\psi$. By \eqref{eq:jet1} applied to $u=g$, $v=\psi$:
\[
[g\circ\psi]_k=L_g\circ\psi_k+A_k,\qquad A_k:=\big[(g-L_g)\circ\psi^{(k-1)}\big]_k .
\]
By \eqref{eq:jet2} applied to $u=\psi$, $v=h$:
\[
[\psi\circ h]_k=\psi_k\circ L+B_k,\qquad B_k:=\big[\psi^{(k-1)}\circ h\big]_k .
\]
Now $g-L_g\in SR(L)$, $h\in SR(L)$ and $\psi^{(k-1)}\in SR(L)$; by stability of $SR(L)$ under composition and extraction of homogeneous components (Remark \ref{r:noinvert}), $A_k,B_k\in SR_k(L)$. Equating the two extractions gives
\begin{equation}\label{eq:homological}
\psi_k\circ L\ -\ L_g\circ\psi_k\ =\ A_k-B_k\ \in\ SR_k(L).
\end{equation}
Set $\chi_k:=A^{-1}\circ\psi_k$. Using \eqref{eq:AL}:
\[
\psi_k\circ L-L_g\circ\psi_k
= A\circ\chi_k\circ L- (A L A^{-1})\circ A\circ\chi_k
= A\circ\big(\chi_k\circ L-L\circ\chi_k\big)
= A\circ M^{k}_{L}(\chi_k).
\]
Since $A^{-1}$ is linear sub-resonant, \eqref{eq:homological} gives
\[
M^{k}_{L}(\chi_k)=A^{-1}\circ(A_k-B_k)\in SR_k(L),
\]
and Lemma \ref{l:stable&ante}(2) implies $\chi_k\in SR_k(L)$. Finally $\psi_k=A\circ\chi_k\in SR_k(L)$ by stability, which closes the induction.

Step 3: conclusion. By Corollary \ref{c:degreebound}, $SR_k(L)=0$ for $k>c_0(L)$, so $\psi_k=0$ for $k>c_0(L)$: the Taylor series of $\psi$ is a polynomial, and the germ $\psi$ coincides with this polynomial. All its homogeneous components are sub-resonant and $D\psi(0)=A$ is invertible: $\psi\in SR^*(L)$. The last assertions follow from Corollaries \ref{c:degreebound} and \ref{c:autpol}.
\end{proof}

\begin{remark}\label{r:hypothesisneeded}
The hypothesis $\FF(L_g)=V_\bullet$ cannot be dropped: for $L=\mathrm{diag}(1/3,1/2)$, $g:=\mathrm{diag}(1/2,1/3)\in SR^*(L)$ and $\psi(z_1,z_2)=(z_2,z_1)$ satisfy $\psi\circ L=g\circ\psi$, but $\psi\notin SR^*(L)$ (it does not preserve $V_1=\C e_1$). It is precisely Step 1 that fails: $A$ intertwines $L$ and $L_g$, but has no reason to preserve $V_\bullet$ if the flag of $L_g$ is not the standard one.
\end{remark}

\begin{proposition}\label{main:aligned}
Let $n\geq 2$. There is exactly one compatible aligned structure on $M$, the canonical structure. Any of its representatives has developing map equal, up to translation, to the restriction of an element of $SR^*(L)$; in particular it is polynomial of degree at most $c_0(L)$.
\end{proposition}

\begin{proof}[Proof of Proposition \ref{main:aligned}]
Let $\Sc$ be an aligned structure, and choose a representative $(\dev,\rho)$. Adopt the notation of the proof of Proposition \ref{main:rigidity}: $p$, $g=t_{-p}\rho(\gamma)t_p\in SR^*(L)$ and $\Psi=t_{-p}\circ\widehat{\dev}\in\Aut(\C^n)$, with $\Psi(0)=0$ and $\Psi\circ h=g\circ\Psi$. Since $D\rho(\gamma)(p)=Dg(0)=L_g$, alignment means $\FF(L_g)=V_\bullet$. The germ of $\Psi$ at $0$ then satisfies all the hypotheses of Lemma \ref{l:alignedrigidity}: there exists $s\in SR^*(L)$ whose germ at $0$ coincides with that of $\Psi$; by analytic continuation, $\Psi=s$ on $\C^n$.

Set $g_0:=t_p\circ s\in G$. Then $\dev=\widehat{\dev}|_{\C^n\setminus\{0\}}=g_0\circ\dev_{\mathrm{can}}$ and
\[
\rho(\gamma)=\widehat{\dev}\circ h\circ\widehat{\dev}^{\,-1}=g_0\circ h\circ g_0^{-1}=g_0\,\rho_{\mathrm{can}}(\gamma)\,g_0^{-1} :
\]
the chosen representative $(\dev,\rho)$ of $\Sc$ is $G$-equivalent to the representative $(\dev_{\mathrm{can}},\rho_{\mathrm{can}})$ of the canonical structure, so $\Sc=\Sc_{\mathrm{can}}$. Since the canonical structure is itself aligned (Lemma \ref{l:aligned-invariant}), it is the unique aligned structure. The description of the developing map follows from $\dev=g_0\circ\dev_{\mathrm{can}}$ with $g_0=t_p\circ s$.
\end{proof}

\section{Classification and optimality}\label{s:classification}

\subsection{Classification of compatible structures}

Set
\[
\Norm(h)=\big\{\Psi\in\Aut(\C^n)\ :\ \Psi(0)=0,\ \ \Psi\, h\,\Psi^{-1}\in SR^*(L)\big\}.
\]
The group $SR^*(L)$ acts on $\Norm(h)$ by left composition: if $s\in SR^*(L)$ and $\Psi\in\Norm(h)$, then $(s\Psi)h(s\Psi)^{-1}=s(\Psi h\Psi^{-1})s^{-1}\in SR^*(L)$. For $\Psi\in\Norm(h)$, we write $\Sc_\Psi$ for the structure represented by the pair $\big(\Psi|_{\C^n\setminus\{0\}},\ \rho_\Psi\big)$, where $\rho_\Psi(\gamma):=\Psi h\Psi^{-1}$.

\begin{proposition}\label{main:classification}
Let $n\geq 2$.
\begin{enumerate}
\item The assignment $\Psi\mapsto\Sc_\Psi$ induces a bijection
\[
SR^*(L)\backslash\Norm(h)\ \xrightarrow{\ \sim\ }\ \big\{\text{compatible $(G,X)$-structures on }M\big\}.
\]
\item The aligned structure corresponds to the class of the identity in $SR^*(L)\backslash\Norm(h)$.
\item If $h=L=\mathrm{diag}(\lambda_1,\dots,\lambda_n)$ with pairwise distinct moduli, then $P_\sigma\in\Norm(h)$ for every $\sigma\in\mathfrak{S}_n$ and the classes $[P_\sigma]$ are pairwise distinct; exactly one of them, $\sigma=\mathrm{id}$, is aligned. If moreover $L$ is non-resonant, these $n!$ structures are pairwise inequivalent modulo $\Aut(M)$ \textup{(}Proposition \ref{p:autM}\textup{)}.
\end{enumerate}
\end{proposition}

\begin{proof}[Proof of Proposition \ref{main:classification}]
(1) The map is well defined. For $\Psi\in\Norm(h)$, $\Psi|_{\C^n\setminus\{0\}}$ is a local biholomorphism and $\rho_\Psi(\gamma)\in SR^*(L)\subset G$, with the equivariance $\Psi\circ h=\rho_\Psi(\gamma)\circ\Psi$: $\Sc_\Psi$ is a compatible structure.

Surjectivity. This is Proposition \ref{main:rigidity}: any structure has a representative $(\dev,\rho)$ with $\dev=t_p\circ\Psi|_{\C^n\setminus\{0\}}$ for some $\Psi\in\Norm(h)$, and this structure equals $\Sc_\Psi$, since $(\dev,\rho)$ and the defining representative of $\Sc_\Psi$ are related by $g_0:=t_p\in G$ (the relation between the holonomies follows as in the proof of Proposition \ref{main:aligned}).

Injectivity / compatibility with the classes. Let $\Psi_1,\Psi_2\in\Norm(h)$. Suppose $\Sc_{\Psi_1}=\Sc_{\Psi_2}$, i.e.\ their defining representatives are $G$-equivalent: there exists $g_0\in G$ with $\Psi_2=g_0\circ\Psi_1$ on $\C^n\setminus\{0\}$, hence on $\C^n$ by continuity. Evaluating at $0$: $0=\Psi_2(0)=g_0(\Psi_1(0))=g_0(0)$, so $g_0\in\Stab_G(0)=SR^*(L)$ (Lemma \ref{l:uniquedecomp}) and $[\Psi_1]=[\Psi_2]$ in $SR^*(L)\backslash\Norm(h)$. Conversely, if $\Psi_2=s\circ\Psi_1$ with $s\in SR^*(L)\subset G$, then $g_0:=s$ shows the defining representatives are $G$-equivalent, since $\rho_{\Psi_2}(\gamma)=s\,\rho_{\Psi_1}(\gamma)\,s^{-1}$ by direct computation; hence $\Sc_{\Psi_1}=\Sc_{\Psi_2}$.

(2) If $\Psi=s\in SR^*(L)$, then $\Sc_\Psi=\Sc_{\mathrm{can}}$ (their representatives are related by $g_0=s$), hence $\Sc_\Psi$ is aligned. Conversely, if $\Sc_\Psi$ is aligned, Proposition \ref{main:aligned} gives $\Sc_\Psi=\Sc_{\mathrm{can}}=\Sc_{\mathrm{id}}$, so $[\Psi]=[\mathrm{id}]$ by (1).

(3) $g_\sigma:=P_\sigma L P_\sigma^{-1}=\mathrm{diag}(\lambda_{\sigma^{-1}(1)},\dots,\lambda_{\sigma^{-1}(n)})$ is a diagonal matrix, hence preserves the coordinate flag $V_\bullet$ (here $V_i=\mathrm{Vect}(e_1,\dots,e_i)$, the moduli being pairwise distinct); by Corollary \ref{c:flag}, $g_\sigma\in SR^*(L)$ and $P_\sigma\in\Norm(h)$. If $[P_\sigma]=[P_{\sigma'}]$, then $s:=P_{\sigma'}P_\sigma^{-1}=P_{\sigma'\sigma^{-1}}\in SR^*(L)$; being linear, $s$ preserves each $V_i=\mathrm{Vect}(e_1,\dots,e_i)$ (Corollary \ref{c:flag}), which for a permutation matrix forces $\sigma'\sigma^{-1}(\{1,\dots,i\})=\{1,\dots,i\}$ for every $i$, hence $\sigma'=\sigma$. Finally, the defining representative of $\Sc_{P_\sigma}$ has holonomy flag $\FF(g_\sigma)=P_\sigma(V_\bullet)$, which equals $V_\bullet$ if and only if $\sigma=\mathrm{id}$; so $\Sc_{P_\sigma}$ is aligned exactly when $\sigma=\mathrm{id}$. The statement modulo $\Aut(M)$ is Proposition \ref{p:autM} below.
\end{proof}

\subsection{Persistence of non-uniqueness modulo $\Aut(M)$}

The following statement shows that the defect of uniqueness cannot be absorbed by the automorphisms of the manifold: two distinct compatible structures may remain distinct even after precomposition by an automorphism of $M$.

We say that two compatible structures $\Sc,\Sc'$ are equivalent modulo $\Aut(M)$ if there exists $f\in\Aut(M)$ with $f^{*}\Sc=\Sc'$ as structures; in terms of representatives, if there exist a lift $\widetilde f\in\Aut(\C^n\setminus\{0\})$ of $f$ and $g_0\in G$ with $\dev'=g_0\circ\dev\circ\widetilde f$, for some representatives $(\dev,\rho)$ of $\Sc$ and $(\dev',\rho')$ of $\Sc'$.

\begin{proposition}\label{p:autM}
Let $h=L=\mathrm{diag}(\lambda_1,\dots,\lambda_n)$, $0<|\lambda_1|<\dots<|\lambda_n|<1$, be non-resonant: $\lambda^{I}\neq\lambda_j$ for every multi-index $I$ with $|I|\ge2$ and every $j$. Then:
\begin{enumerate}
\item every lift to $\C^n\setminus\{0\}$ of an automorphism of $M$ is the restriction of an invertible diagonal linear map;
\item the structures $\Sc_{P_\sigma}$, $\sigma\in\mathfrak S_n$, are pairwise inequivalent modulo $\Aut(M)$.
\end{enumerate}
\end{proposition}

\begin{proof}
(1) Let $f\in\Aut(M)$ and let $\widetilde f$ be a lift to the universal cover $\C^n\setminus\{0\}$. Compatibility with the deck group reads $\widetilde f\circ L\circ\widetilde f^{-1}=L^{\eps}$ for some $\eps\in\{\pm1\}$ (conjugation by $\widetilde f$ is an automorphism of $\langle L\rangle\simeq\Z$, hence sends the generator to $L^{\pm1}$). By Hartogs ($n\ge2$), $\widetilde f$ extends to a holomorphic map $\widehat f:\C^n\to\C^n$, and the relation $\widehat f\circ L=L^{\eps}\circ\widehat f$ extends by analytic continuation. Then $\widehat f(0)$ is a fixed point of $L^{\eps}$; since $\Fix(L)=\Fix(L^{-1})=\{0\}$, we get $\widehat f(0)=0$. The Jacobian argument of the proof of Proposition \ref{main:rigidity} (applied verbatim to $\widehat f$, which is a local biholomorphism on $\C^n\setminus\{0\}$) shows that $B:=D\widehat f(0)$ is invertible; differentiating at $0$: $B\,L=L^{\eps}B$, so $L$ and $L^{\eps}$ are similar. If $\eps=-1$, the eigenvalues of $L^{-1}$ have modulus $>1$: impossible. So $\eps=1$ and $\widehat f$ commutes with $L$.

Expand $\widehat f=\sum_{k\ge1}\widehat f_k$ (no constant term since $\widehat f(0)=0$); commutation gives, in degree $k$, $\widehat f_k\circ L=L\circ\widehat f_k$. In the monomial basis $H_{I,j}$, this equation reads $(\lambda^{I}-\lambda_j)\,c_{I,j}=0$ for each coefficient $c_{I,j}$ of $\widehat f_k$. For $k=|I|\ge2$, non-resonance gives $c_{I,j}=0$: $\widehat f=B$ is linear. Finally $BL=LB$ with $L$ diagonal with pairwise distinct eigenvalues (their moduli are distinct) forces $B$ diagonal.

(2) Suppose $\Sc_{P_{\sigma'}}$ is equivalent to $\Sc_{P_\sigma}$ modulo $\Aut(M)$: there exist a lift $\widetilde f$ and $g_0\in G$ with $P_{\sigma'}=g_0\circ P_\sigma\circ\widetilde f$ on $\C^n\setminus\{0\}$, hence on $\C^n$. By (1), $\widetilde f=D|_{\C^n\setminus\{0\}}$ with $D$ diagonal invertible. Evaluating at $0$: $0=g_0(P_\sigma D\,(0))=g_0(0)$, so $g_0\in SR^*(L)$ (Lemma \ref{l:uniquedecomp}), and
\[
g_0=P_{\sigma'}\,(P_\sigma D)^{-1}=P_{\sigma'}\,D^{-1}P_\sigma^{-1}=P_{\sigma'\sigma^{-1}}\cdot\big(P_\sigma D^{-1}P_\sigma^{-1}\big),
\]
the product of a permutation matrix and a diagonal matrix: it is a monomial matrix with underlying permutation $\sigma'\sigma^{-1}$. Being linear and in $SR^*(L)$, $g_0$ preserves each $V_i=\mathrm{Vect}(e_1,\dots,e_i)$ (Corollary \ref{c:flag}), which forces $\sigma'\sigma^{-1}=\mathrm{id}$, i.e.\ $\sigma=\sigma'$.
\end{proof}

\section{Secondary Hopf manifolds}\label{s:secondary}

So far $M=W_h=(\C^n\setminus\{0\})/\langle h\rangle$ has been a primary Hopf manifold. We now show that the construction extends, without modification of the model $(G,X)$, to every secondary Hopf manifold, and we isolate the algebraic reason why.

Let $X$ be a secondary Hopf manifold of dimension $n\geq2$. By definition its universal cover is $\C^n\setminus\{0\}$, and $X=(\C^n\setminus\{0\})/\Gamma$ with $\Gamma=\pi_1(X)\subset\Aut(\C^n\setminus\{0\})$ acting freely and properly discontinuously; by the theorem of Kodaira in dimension two and Hasegawa \cite{Hasegawa} in general, $\Gamma$ contains a normal infinite cyclic subgroup $\langle\gamma_0\rangle$ of finite index generated by a contraction. Applying the Poincar\'e--Dulac normalization (Theorem \ref{t:PDSR}) to $\gamma_0$, that is, replacing $\Gamma$ by a conjugate under a biholomorphism of $\C^n\setminus\{0\}$ fixing the puncture, we may assume that $\gamma_0=h\in SR^*(L)$ is in normal form, with $D h(0)=L$. Thus $\langle h\rangle\lhd\Gamma$ is normal of finite index, and $F:=\Gamma/\langle h\rangle$ is the finite group of the covering $W_h\to X$.

We first record the rigidity of the elements of $\Gamma$, extending to non-contracting elements the arguments already used for $h$ and for $\Aut(M)$ (Proposition \ref{p:autM}). The following lemma, and more generally the structure of $\Gamma$ as a central extension of a finite group by $\langle h\rangle$, is presumably well known; we include a proof for completeness and to keep the paper self-contained.

\begin{lemma}\label{l:secondary-lift}
Every $g\in\Gamma$ extends to a polynomial automorphism of $\C^n$ fixing $0$, and satisfies $g\circ h=h\circ g$. In other words, $h$ is central in $\Gamma$.
\end{lemma}

\begin{proof}
Let $g\in\Gamma$. Its components are holomorphic on $\C^n\setminus\{0\}$ with $n\geq2$, so by Hartogs' theorem $g$ extends to a holomorphic map $\widehat g:\C^n\to\C^n$; the same applies to $g^{-1}$, and the relations $\widehat g\circ\widehat{g^{-1}}=\widehat{g^{-1}}\circ\widehat g=\mathrm{id}$, valid on the dense open set $\C^n\setminus\{0\}$, extend by continuity, so $\widehat g\in\Aut(\C^n)$.

Since $\langle h\rangle$ is normal in $\Gamma$, there is $\eps\in\{\pm1\}$ with $g\, h\, g^{-1}=h^{\eps}$ (the only automorphisms of $\Z$ being $\pm1$); this identity of automorphisms of $\C^n\setminus\{0\}$ extends to $\C^n$. Evaluating at $0$ gives $\widehat g(0)=\widehat g(h(0))=h^{\eps}(\widehat g(0))$, so $\widehat g(0)$ is a fixed point of $h^{\eps}$. As in the proof of Proposition \ref{main:rigidity}, the Jacobian of $\widehat g$ cannot vanish at $0$ (otherwise $1/\det D\widehat g$ would extend across $0$ by Hartogs, contradicting $\det D\widehat g\not\equiv0$), so $\widehat g$ is a local biholomorphism at $0$; since $\Fix(h)=\Fix(h^{-1})=\{0\}$ by Lemma \ref{l:dynamics}, we get $\widehat g(0)=0$. Differentiating $g\, h\, g^{-1}=h^{\eps}$ at $0$ yields $B\,L\,B^{-1}=L^{\eps}$ with $B:=D\widehat g(0)$ invertible; hence $L$ and $L^{\eps}$ are similar and have the same spectrum. If $\eps=-1$, the eigenvalues of $L^{-1}$ have modulus $>1$ while those of $L$ have modulus $<1$: impossible. Therefore $\eps=+1$, that is $g\,h\,g^{-1}=h$.
\end{proof}

\begin{proposition}\label{main:secondary}
Let $X$ be a secondary Hopf manifold of dimension $n\geq 2$, written $X=(\C^n\setminus\{0\})/\Gamma$ with $\Gamma\subset\Aut(\C^n\setminus\{0\})$ containing a normal cyclic contraction subgroup $\langle h\rangle$ of finite index, $h$ a Poincar\'e--Dulac normal form. Then $\Gamma\subset SR^*(L)\subset G$; consequently the inclusion $\C^n\setminus\{0\}\hookrightarrow\C^n$ is a developing map for a compatible $(G,X)$-structure on $X$, with holonomy the inclusion $\Gamma\hookrightarrow G$. Together with the primary case, this completes the proof of Theorem \ref{main:existence}.
\end{proposition}

\begin{proof}[Proof of Proposition \ref{main:secondary}]
Let $g\in\Gamma$. By Lemma \ref{l:secondary-lift}, $g$ is a biholomorphism of $\C^n$ fixing $0$ with $g\circ h=h\circ g$. Apply the aligned rigidity lemma (Lemma \ref{l:alignedrigidity}) with the following choice of data: the r\^ole of ``$h$'' is played by $h$, the r\^ole of ``$g$'' is also played by $h$, and the r\^ole of ``$\psi$'' by $g$. The hypotheses hold: $h\in SR^*(L)$ with $Dh(0)=L$; the flag condition $\FF(L_g)=V_\bullet$ becomes $\FF(L)=V_\bullet$, which is true by definition of $V_\bullet$; $g$ is a local biholomorphism at $0$ fixing $0$; and the intertwining relation ``$\psi\circ h=g\circ\psi$'' becomes $g\circ h=h\circ g$, which holds. The lemma yields $g\in SR^*(L)$.

Since $g\in\Gamma$ was arbitrary, $\Gamma\subset SR^*(L)\subset G$. The group $\Gamma$ acts freely and properly discontinuously on the open subset $\C^n\setminus\{0\}$ of $X_{\mathrm{model}}=\C^n$, on which $G$ acts transitively and holomorphically; by Proposition \ref{p:structure quotient}, the quotient $X=(\C^n\setminus\{0\})/\Gamma$ inherits a holomorphic $(G,X)$-structure. Its charts are local inverses of the covering map followed by the inclusion $\C^n\setminus\{0\}\hookrightarrow\C^n$, hence biholomorphisms for the complex structure of $X$: the structure is compatible, with developing map the inclusion $\C^n\setminus\{0\}\hookrightarrow\C^n$ and holonomy the inclusion $\Gamma\hookrightarrow SR^*(L)\subset G$. This settles the secondary case, and with the primary case of Section \ref{s:dynamics} completes the proof of Theorem \ref{main:existence}.
\end{proof}

\begin{remark}
The proof shows that the decisive fact is the centrality of $h$ in $\Gamma$, not merely its normality: it is centrality that turns the normalization relation into the commutation $g\circ h=h\circ g$, to which aligned rigidity applies with the flag condition automatically satisfied. In concrete terms, $D\widehat g(0)$ commutes with $L$, hence preserves every characteristic subspace $E^\lambda(L)$ and therefore every $V_i=\bigoplus_{|\lambda|\le\mu_i}E^\lambda(L)$: the modulus flag $V_\bullet$ is preserved by the whole deck group, with no genericity or non-resonance assumption. This is why no secondary quotient can obstruct the descent of the structure, and why, in contrast with the failure of marked uniqueness on primary manifolds, there is here no combinatorial room for a counterexample: the permutation phenomenon of Example \ref{ex:permutations} arises from developing maps that do not commute with $h$, whereas deck transformations necessarily do.
\end{remark}

\begin{remark}
The uniqueness results of Sections \ref{s:rigidity}--\ref{s:classification} also descend, in the following sense. A compatible structure on $X$ pulls back to a $\Gamma$-invariant, hence a fortiori $\langle h\rangle$-invariant, compatible structure on the primary cover $W_h$, i.e.\ to a compatible structure on $W_h$ whose holonomy representation extends from $\langle h\rangle$ to $\Gamma$. Thus the compatible structures on $X$ correspond to the $F$-equivariant compatible structures on $W_h$, a subfamily of the classification of Proposition \ref{main:classification}; the aligned structure, being canonical, is $F$-equivariant and descends to the distinguished compatible structure on $X$ produced by Proposition \ref{main:secondary}.
\end{remark}

\section{Comparison with the work of Madera}\label{s:madera}
The recent work of Madera \cite{Madera} addresses a closely related question by a converse route, and it is worth spelling out where the two approaches meet and where they diverge. Madera constructs, on every primary Hopf manifold, holomorphic $G$-structures of higher order in the sense of Kobayashi (principal subbundles of the frame bundle $R^k(M)$) and shows that they induce flat holomorphic Cartan geometries. His structure group $G^r_\beta$, the group of sub-resonant polynomial transformations of maximal degree $r$ attached to the eigenvalue data $\beta$, coincides with the group $SR^*(L)$ used here, and the Lie group $\widetilde{G^r_\beta}$ generated by $G^r_\beta$ together with the translations of $\C^n$ is exactly our model group $G$. His Lemma 7.1.1, stating that $\theta_{-P(z)}\circ P\circ\theta_z\in G^r_\beta$, is the same computation as the one underlying our Lemma \ref{l:Gform}.

The methods, however, run in opposite directions, and this is the substance of the comparison. Madera works without the Poincar\'e--Dulac theorem: starting from a contraction in arbitrary form, he builds an order-one structure by a cohomological argument on Mall bundles (the pullback operator $\gamma^*$ on bounded holomorphic forms is compact, and Riesz--Schauder applies), extends it order by order using the vanishing of certain first cohomology groups, and proves integrability through a Frobenius-type theorem of Benoist. From the resulting integrable structure he then derives the Poincar\'e--Dulac normal form, as the change of coordinates furnished by an integrability chart. Our route is the reverse: we take the Poincar\'e--Dulac normalization as the starting point (giving it, in Sections \ref{s:sr}--\ref{s:pd}, a self-contained treatment through the sub-resonant operators $M^k_L$), and read off the $(G,X)$-structure directly from the normal form. Thus Madera's flat Cartan geometry and our locally homogeneous $(G,X)$-structure are two faces of the same underlying object: an integrable $G^r_\beta$-structure of order $r$ and a $(\widetilde{G^r_\beta},\C^n)$-structure carry the same information, and each of the two papers can be seen as recovering the input of the other as an output. In particular, on the existence question the two results imply one another, once the dictionary $G^r_\beta=SR^*(L)$, $\widetilde{G^r_\beta}=G$ is in place; the developing map of our structure is precisely the integrability chart of his.

Beyond this common core, the two papers pursue different aims, and neither subsumes the other. Madera's emphasis is on the order-by-order construction and on the Cartan-geometric and Gromov-rigidity framework, including a careful analysis of Hopf surfaces, where a cohomological obstruction forces the structure to be extended only up to order $r+2$. Our emphasis is on the developing-map description and on uniqueness: the notion of aligned structure, the classification of compatible structures by the coset space $SR^*(L)\backslash\Norm(h)$, and the failure of marked uniqueness (the family of Example \ref{ex:permutations}) have no counterpart in \cite{Madera}. Conversely, the higher-order and Cartan-geometric refinements, and the Mall-bundle proof of existence bypassing normal forms, are specific to \cite{Madera}. Finally, Proposition \ref{main:secondary} extends the existence statement to secondary Hopf manifolds, which lie outside the scope of \cite{Madera} (its standing definition requires an infinite cyclic fundamental group).

\end{document}